\documentclass[11pt,fleqn]{article}
\usepackage{geometry}                
\usepackage{graphicx}
\usepackage{amsthm,amsmath}
\usepackage{amssymb,dsfont,mathrsfs}

\numberwithin{equation}{section}
\newtheorem{theorem}{Theorem}[section]
\newtheorem{lemma}{Lemma}[section]
\newtheorem{proposition}{Proposition}[section]
\newtheorem{corollary}{Corollary}[section]
\newtheorem{remark}{Remark}[section]
\newtheorem{assumption}{Assumption}[section]

\newtheorem{condition}{Condition}[section]
\def\ba{\boldsymbol{a}}
\def\bb{\boldsymbol{b}}
\def\bc{\boldsymbol{c}}

\def\bi{\boldsymbol{i}}

\def\bk{\boldsymbol{k}}
\def\bl{\boldsymbol{l}}
\def\bm{\boldsymbol{m}}
\def\bq{\boldsymbol{q}}

\def\bu{\boldsymbol{u}}
\def\bv{\boldsymbol{v}}

\def\bx{\boldsymbol{x}}
\def\by{\boldsymbol{y}}
\def\bz{\boldsymbol{z}}
\def\bJ{\boldsymbol{J}}
\def\bM{\boldsymbol{M}}
\def\bX{\boldsymbol{X}}
\def\bY{\boldsymbol{Y}}
\def\bZ{\boldsymbol{Z}}

\def\bpi{\boldsymbol{\pi}}
\def\bnu{\boldsymbol{\nu}}
\def\btheta{\boldsymbol{\theta}}

\def\bzero{\mathbf{0}}
\def\bone{\mathbf{1}}

\def\calD{\mathcal{D}}

\def\calS{\mathcal{S}}

\def\scrI{\mathscr{I}}

\def\cp{\mbox{\rm cp}}
\def\diag{\mbox{\rm diag}}

\title{Asymptotic property of the occupation measures in a multidimensional skip-free Markov modulated random walk}
\author{Toshihisa Ozawa  \\ 
Faculty of Business Administration, Komazawa University \\
1-23-1 Komazawa, Setagaya-ku, Tokyo 154-8525, Japan \\
E-mail: toshi@komazawa-u.ac.jp
}
\date{}

\begin{document}

\maketitle

\begin{abstract}
We consider a discrete-time $d$-dimensional process $\{\bX_n\}=\{(X_{1,n},X_{2,n},...,X_{d,n})\}$ on $\mathbb{Z}^d$ with a background process $\{J_n\}$ on a countable set $S_0$, where individual processes $\{X_{i,n}\},i\in\{1,2,...,d\},$ are skip free. We assume that the joint process $\{\bY_n\}=\{(\bX_n,J_n)\}$ is Markovian and that the transition probabilities of the $d$-dimensional process $\{\bX_n\}$ vary according to the state of the background process $\{J_n\}$. This modulation is assumed to be space homogeneous. We refer to this process as a $d$-dimensional skip-free Markov modulate random walk. 
For $\by, \by'\in \mathbb{Z}_+^d\times S_0$, consider the process $\{\bY_n\}_{n\ge 0}$ starting from the state $\by$ and let $\tilde{q}_{\by,\by'}$ be the expected number of visits to the state $\by'$ before the process leaves the nonnegative area $\mathbb{Z}_+^d\times S_0$ for the first time. 
For $\by=(\bx,j)\in \mathbb{Z}_+^d\times S_0$, the measure $(\tilde{q}_{\by,\by'}; \by'=(\bx',j')\in \mathbb{Z}_+^d\times S_0)$ is called  an occupation measure. Our primary aim is to obtain the asymptotic decay rate of the occupation measure as $\bx'$ go to infinity in a given direction. We also obtain the convergence domain of the matrix moment generating function of the occupation measures.

\smallskip
{\it Key wards}: Markov modulated random walk, Markov additive process, occupation measure, asymptotic decay rate, moment generating function, convergence domain

\end{abstract}

%
%
\section{Introduction} \label{sec:intro}

%
For $d\ge 1$, we consider a discrete-time $d$-dimensional process $\{\bX_n\}=\{(X_{1,n},X_{2,n},...,X_{d,n})\}$ on $\mathbb{Z}^d$, where $\mathbb{Z}$ is the set of all integers, and a background process $\{J_n\}$ on a countable set $S_0=\{1,2,...\}$. We assume that each individual process $\{X_{i,n}\}$ is skip free, which means that its increments take values in $\{-1,0,1\}$. Furthermore, we assume that the joint process $\{\bY_n\}=\{(\bX_n,J_n)\}$ is Markovian and that the transition probabilities of the $d$-dimensional process $\{\bX_n\}$ vary according to the state of the background process $\{J_n\}$. This modulation is assumed to be space homogeneous. 
We refer to this process as a $d$-dimensional skip-free Markov modulate random walk (MMRW for short). The state space of the $d$-dimensional MMRW is given by $\mathbb{S}=\mathbb{Z}^d\times S_0$.
It is also a $d$-dimensional Markov additive process (MA-process for short) \cite{Miyazawa12}, where $\bX_n$ is the additive part and $J_n$ the background state. 
A discrete-time $d$-dimensional quasi-birth-and-death process  \cite{Ozawa13} (QBD process for short) is a $d$-dimensional MMRW with reflecting boundaries, where the process $\bX_n$ is the level and $J_n$ the phase. 
Stochastic models arising from various Markovian multiqueue models and queueing networks such as polling models and generalized Jackson networks with Markovian arrival processes and phase-type service processes can be represented as \textit{continuous-time} multidimensional QBD processes (in the case of two-dimension, see, e.g., \cite{Miyazawa15} and \cite{Ozawa13,Ozawa18}) and, by using the uniformization technique, they can be deduced to \textit{discrete-time} multidimensional QBD processes. 
%
%
It is well known that, in general,  the stationary distribution of a Markov chain can be represented in terms of its stationary probabilities on some boundary faces and its occupation measures. In the case of multidimensional QBD process, such occupation measures are given as those in the corresponding multidimensional MMRW. For this reason, we focus on multidimensional MMRWs and study their occupation measures, especially, asymptotic properties of the occupation measures. 
Here we briefly explain that the assumption of skip-free is not so restricted. For a given $k>1$, assume that, for $i\in\{1,2,...,d\}$, $X_{i,n}$ takes values in $\{-k,-(k-1),...,0,1,...,k\}$. For $i\in\{1,2,...,d\}$, let ${}^k\!X_{i,n}$ and ${}^k\!M_{i,n}$ be the quotient and remainder of $X_{i,n}$ divided by $k$, respectively, where ${}^k\!X_{i,n}\in\mathbb{Z}$ and $0\le {}^k\!M_{i,n} \le k-1$.  Then, the process $\{( {}^k\!X_{1,n},...,{}^k\!X_{d,n},({}^k\!M_{1,n},...,{}^k\!M_{d,n},J_n))\}$ becomes a $d$-dimensional MMRW \textit{with skip-free jumps}, where $({}^k\!X_{1,n},...,{}^k\!X_{d,n})$ is the level and $({}^k\!M_{1,n},...,{}^k\!M_{d,n},J_n)$ the background state. This means that any multidimensional MMRW \textit{with bounded jumps} can be reduced to a multidimensional MMRW \textit{with skip-free jumps}. 

%
Let $P=\left(p_{(\bx,j),(\bx',j')}; (\bx,j),(\bx',j')\in\mathbb{S}\right)$ be the transition probability matrix of the $d$-dimensional MMRW $\{\bY_n\}$, where $p_{(\bx,j)(\bx',j')}=\mathbb{P}(\bY_1=(\bx',j')\,|\,\bY_0=(\bx,j))$. 
By the property of skip-free, each element of $P$, say $p_{(\bx,j)(\bx',j')}$, is nonzero only if $\bx_1'-\bx_1\in\{-1,0,1\}^d$. By the property of space-homogeneity, for every $\bx,\bx'\in\mathbb{Z}^d$, $\bi\in\{-1,0,1\}^d$ and $j,j'\in S_0$, we have $p_{(\bx,j),(\bx+\bi,j')}=p_{(\bx',j),(\bx'+\bi,j')}$. 
Hence, the transition probability matrix $P$ can be represented as a block matrix in terms of only the following blocks:
\[
A_{\bi}=\left(p_{(\bzero,j)(\bi,j')}; j,j'\in S_0\right),\ \bi\in\{-1,0,1\}^d, 
\]
i.e., for $\bx,\bx'\in\mathbb{Z}^d$, block $P_{\bx,\bx'}=(p_{(\bx,j)(\bx',j')}; \,j,j'\in S_0)$ is given as 
\begin{equation}
P_{\bx,\bx'}
= \left\{ \begin{array}{ll}
A_{\bx'-\bx}, & \mbox{if $\bx'-\bx\in\{-1,0,1\}$}^d, \cr
O, & \mbox{otherwise}, 
\end{array} \right.
\end{equation}
where $\bzero$ and $O$ are a vector and matrix of $0$'s, respectively, whose dimensions are determined in context. 
Define a set $\mathbb{S}_+$ as $\mathbb{S}_+=\mathbb{Z}_+^d\times S_0$, where $\mathbb{Z}_+$ is the set of all nonnegative integers, and let $\tau$ be the stopping time at which the MMRW $\{\bY_n\}$ enters $\mathbb{S}\setminus\mathbb{S}_+$ for the first time, i.e., 
\[
\tau =\inf\{n\ge 0; \bY_n\in\mathbb{S}\setminus\mathbb{S}_+\}.
\]
For $\by=(\bx,j),\by'=(\bx',j')\in\mathbb{S}_+$, let $\tilde{q}_{\by,\by'}$ be the expected number of visits to the state $\by'$ before the process $\{\bY_n\}$ starting from the state $\by$ enters $\mathbb{S}\setminus\mathbb{S}_+$ for the first time, i.e.,
\begin{equation}
\tilde{q}_{\by,\by'} = \mathbb{E}\bigg(\sum_{n=0}^\infty 1\big(\bY_n=\by'\big)\, 1(\tau>n)\,\Big|\,\bY_0=\by \bigg), 
\label{eq:tildeq_def}
\end{equation}
where $1(\cdot)$ is an indicator function. For $\by\in\mathbb{S}_+$, the measure $(\tilde{q}_{\by,\by'}; \by'\in\mathbb{S}_+)$ is called an occupation measure. 
Note that $\tilde{q}_{\by,\by'}$ is the $(\by,\by')$-element of the fundamental matrix of the truncated substochastic matrix $P_+$ given as $P_+=\left(p_{\by,\by'}; \by,\by'\in\mathbb{S}_+\right)$, i.e., $\tilde{q}_{\by,\by'} = [\tilde{P}_+]_{\by,\by'}$ and 
\[
\tilde{P_+}
=\sum_{k=0}^\infty P_+^k,
\]
where, for example, $P_+^2=\left(p^{(2)}_{\by,\by'}\right)$ is defined by $p^{(2)}_{\by,\by'} = \sum_{\by''\in\calS_+} p_{\by,\by''}\,p_{\by'',\by'}$. $P_+$ governs transitions of $\{\bY_n\}$ on $\mathbb{S}_+$. 
Our primary aim is to obtain the asymptotic decay rate of the occupation measure $(\tilde{q}_{\by,\by'}; \by'=(\bx',j')\in\mathbb{S}_+)$ as $\bx'$ goes to infinity in a given direction. 
This asymptotic decay rate gives a lower bound for the asymptotic decay rate of the stationary distribution in a corresponding multidimensional QBD process in the same direction. Such lower bounds have been obtained for some kinds of multidimensional reflected process \textit{without background states}; for example, $0$-partially chains in \cite{Borovkov01}, also see comments on Conjecture 5.1 in \cite{Miyazawa12}. 
With respect to multidimensional reflected processes \textit{with background states}, such asymptotic decay rates of the stationary tail distributions in two-dimensional reflected processes have been discussed  in \cite{Miyazawa12,Miyazawa15} by using Markov additive processes and large deviations. Note that the asymptotic decay rates of the stationary distribution in a two-dimensional QBD process with finite phase states \textit{in the coordinate directions} have been obtained in \cite{Ozawa13,Ozawa18}.

%
As mentioned above, the $d$-dimensional MMRW $\{\bY_n\}=\{(\bX_n,J_n)\}$ is a $d$-dimensional MA-process, where the set of blocks, $\{ A_{\bi}; \bi\in\{-1,0,1\}^d \}$, corresponds to the kernel of the MA-process. 
For $\btheta\in\mathbb{R}^d$, let $A_*(\btheta)$ be the matrix moment generating function of one-step transition probabilities defined as
\begin{equation}
A_*(\btheta) = \sum_{\bi\in\{-1,0,1\}^d} e^{\langle \bi,\btheta \rangle} A_{\bi}, 
\label{eq:As}
\end{equation}
where $\langle \ba,\bb \rangle$ is the inner product of vectors $\ba$ and $\bb$. $A_*(\btheta)$ is the Feynman-Kac operator \cite{Ney87} for the MA-process. 
For $\bx,\bx'\in\mathbb{Z}_+^d$, define a matrix $N_{\bx,\bx'}$ as $N_{\bx,\bx'}=(\tilde{q}_{(\bx,j),(\bx',j')}; j,j'\in S_0)$ and $N_{\bx}$ as $N_{\bx}=(N_{\bx,\bx''}; \bx''\in\mathbb{Z}_+^d)$.  $\tilde{P}_+$ is represented as $\tilde{P}_+=(N_{\bx,\bx'}; \bx,\bx'\in\mathbb{Z}_+^d)$. 
For $\bx\in\mathbb{Z}_+^d$, let $\Phi_{\bx}(\btheta)$ be the matrix moment generating function of the occupation measures defined as 
\[
\Phi_{\bx}(\btheta)=\sum_{\bk\in\mathbb{Z}_+^d} e^{\langle \bk,\btheta \rangle} N_{\bx,\bk},
\]
which satisfies, for $j,j'\in S_0$, 
\begin{equation}
[\Phi_{\bx}(\btheta)]_{j,j'} = \mathbb{E}\bigg( \sum_{n=0}^\infty e^{\langle \bX_n,\btheta \rangle}\, 1(J_n=j')\, 1(\tau>n) \,\Big|\, \bY_0=(\bx,j) \bigg).
\end{equation}
For $\bx\in\mathbb{Z}_+^d$, define the convergence domain of the vector generating function $\Phi_{\bx}(\btheta)$ as
\[
\calD_{\bx} = \mbox{the interior of }\{\btheta\in\mathbb{R}^d; \Phi_{\bx}(\btheta)<\infty\}.
\]
Define point sets $\Gamma$ and $\calD$ as 
\begin{align*}
&\Gamma = \left\{\btheta\in\mathbb{R}^d; \cp(A_*(\btheta))>1 \right\}, \\
&\calD = \left\{\btheta\in\mathbb{R}^d; \mbox{there exists $\btheta'\in\Gamma$ such that $\btheta<\btheta'$} \right\}, 
\end{align*}
where $\cp(A)$ is the convergence parameter of matrix $A$. 
In the following sections, we prove that, for any nonzero vector $\bc\in\mathbb{Z}_+^d$ and for every $j,j'\in S_0$, 
\[
\lim_{k\to\infty} \frac{1}{k} \log \tilde{q}_{(\bzero,j),(k \bc,j')} = -\sup_{\btheta\in\Gamma} \langle \bc,\btheta \rangle.  
\]
Furthermore, using this asymptotic property, we also prove that, for any $\bx\in\mathbb{Z}_+^d$, $\calD_{\bx}$ is given by $\calD$.
In order to obtain these results, we use the matrix analytic method in Queueing theory by extending them to the case where the phase space is countably infinite. Especially, we give a certain expression for the convergence parameter of a nonnegative block tridiagonal matrix with a countable phase space and use it frequently.

%
The rest of the paper is organized as follows.  
In Sect.\ \ref{sec:RandGmatrix}, we extend some results in the matrix analytic method.
In Sect.\ \ref{sec:model}, we introduce some assumptions and give some properties of MMRWs, including a sufficient condition for the occupation measures in a $d$-dimensional MMRW to be finite.
In Sect.\ \ref{sec:QBDrepresentation}, we consider a kind of one-dimensional QBD process with countably many phases and obtain an upper bound for the convergence parameter of the rate matrix in the QBD process. 
Using the upper bound, we obtain the asymptotic decay rates of the occupation measures and the convergence domains of the matrix moment generating functions in Sect.\ \ref{sec:asymptotic}. In the same section, we also consider a single-server polling model with limited services and give some numerical examples. 
The paper concludes with a remark on an asymptotic property of multidimensional QBD processes in Sect.\ \ref{sec:concluding}. 

%
\bigskip
\textit{Notation for matrices.}\quad 
%
%
For a matrix $A$, we denote by $[A]_{i,j}$ the $(i,j)$-element of $A$. 
The transpose of a matrix $A$ is denoted by $A^\top$. 
The convergence parameter of a nonnegative matrix $A$ with a finite or countable dimension is denoted by $\cp(A)$, i.e., $\cp(A) = \sup\{r\in\mathbb{R}_+; \sum_{n=0}^\infty r^n A^n<\infty \}$. 
%
%

%
%
\section{Nonnegative block tridiagonal matrix and its properties} \label{sec:RandGmatrix}

Note that this section is described independently of the following sections. Our aim in the section is to give an expression for the convergence parameter of a nonnegative block tridiagonal matrix whose block size is countably infinite. The role of the Perron-Frobenius eigenvalue of a nonnegative matrix with a finite dimension is replaced with the reciprocal of the convergence parameter of a nonnegative matrix with a countable dimension. 

Consider a nonnegative block tridiagonal matrix $Q$ defined as 
\[
Q 
= \begin{pmatrix}
A_0 & A_1 & & & \cr
A_{-1} & A_0 & A_1 & & \cr
& A_{-1} & A_0 & A_1 & \cr
& & \ddots & \ddots & \ddots 
\end{pmatrix}, 
\]
where $A_{-1}$, $A_0$ and $A_1$ are nonnegative square matrices with a countable dimension, i.e., for $k\in\{-1,0,1\}$, $A_k=(a_{k,i,j}; i,j\in\mathbb{Z}_+)$ and every $a_{k,i,j}$ is nonnegative. 
We define a matrix $A_*$ as 
\[
A_*=A_{-1}+A_0+A_1. 
\]
%
%
Hereafter, we adopt the policy to give a minimal assumption in each place. First, we give the following conditions. 
\begin{condition} \label{cond:As_finite} 
\begin{itemize}
\item[(a1)] Both $A_{-1}$ and $A_1$ are nonzero matrices.
\end{itemize}
\end{condition}

\begin{condition}
\begin{itemize}
\item[(a2)] All iterates of $A_*$ are finite, i.e., for any $n\in\mathbb{Z}_+$, $A_*^n<\infty$. 
\end{itemize}
\end{condition}

Condition (a1) makes $Q$ a \textit{true block tridiagonal} matrix. 
Under condition (a2), all multiple products of $A_{-1}$, $A_0$ and $A_1$ becomes finite, i.e., for any $n\in\mathbb{N}$ and for any $\bi_{(n)}=(i_1,i_2, ..., i_n)\in\{-1,0,1\}^n$, $A_{i_1} A_{i_2} \cdots A_{i_n}<\infty$. 
Hence, for the triplet $\{ A_{-1}, A_0,  A_1\}$, we can define a matrix $R$ corresponding to the rate matrix of a QBD process and a matrix $G$ corresponding to the G-matrix. 
If $\cp(A_*)<\infty$, discussions for $Q$ may be reduced to probabilistic arguments. For example, if there exist an $s>0$ and positive vector $\bv$ such that $s A_* \bv\le \bv$, then $\Delta_{\bv}^{-1} s A_* \Delta_{\bv}$ becomes stochastic or substochastic, where $\Delta_{\bv}=\diag\, \bv$, and discussion for the triplet $\{ A_{-1}, A_0,  A_1\}$ can be replaced with that for $\{ \Delta_{\bv}^{-1}s A_{-1}\Delta_{\bv}, \Delta_{\bv}^{-1}s A_0\Delta_{\bv},  \Delta_{\bv}^{-1}s A_1\Delta_{\bv} \}$. However, in order to make discussion simple, we directly treat $\{ A_{-1}, A_0,  A_1 \}$ and do not use probabilistic arguments. 

Define the following sets of index sequences: for $n\ge 1$ and for $m\ge 1$, 
\begin{align*}
&\scrI_n = \biggl\{\bi_{(n)}\in\{-1,0,1\}^n;\ \sum_{l=1}^k i_l\ge 0\ \mbox{for $k\in\mathbb{N}_{n-1}$}\ \mbox{and} \sum_{l=1}^n i_l=0 \biggr\}, \\
&\scrI_{D,m,n} = \biggl\{\bi_{(n)}\in\{-1,0,1\}^n;\ \sum_{l=1}^k i_l\ge -m+1\ \mbox{for $k\in\mathbb{N}_{n-1}$}\ \mbox{and} \sum_{l=1}^n i_l=-m \biggr\}, \\
&\scrI_{U,m,n} = \biggl\{\bi_{(n)}\in\{-1,0,1\}^n;\ \sum_{l=1}^k i_l\ge 1\ \mbox{for $k\in\mathbb{N}_{n-1}$}\ \mbox{and} \sum_{l=1}^n i_l=m \biggr\}, 
\end{align*}
where $\bi_{(n)}=(i_1,i_2,...,i_n)$ and $\mathbb{N}_{n-1}=\{1,2,...,n-1\}$. Consider a QBD process $\{(X_n,J_n)\}$ on the state space $\mathbb{Z}_+^2$, where $X_n$ is the level and $J_n$ the phase. The set $\scrI_n $ corresponds to the set of all paths of the QBD process on which $X_0=l>0$, $X_k\ge l$ for $k\in\mathbb{N}_{n-1}$ and $X_n=l$, i.e., the level process visits state $l$ at time $n$ without entering states less than $l$ before time $n$. 
The set $\scrI_{D,m,n}$ corresponds to the set of all paths on which $X_0=l>m$, $X_k\ge l-m+1$ for $k\in\mathbb{N}_{n-1}$ and $X_n=l-m$, and $\scrI_{U,m,n}$ to that of all paths on which $X_0=l>0$, $X_k\ge l+1$ for $k\in\mathbb{N}_{n-1}$ and $X_n=l+m$. 
For $n\ge 1$, define $Q_{0,0}^{(n)}$, $D^{(n)}$ and $U^{(n)}$ as 
\begin{align*}
&Q_{0,0}^{(n)} = \sum_{\bi_{(n)}\in\scrI_n} A_{i_1} A_{i_2} \cdots A_{i_n},\quad
D^{(n)} = \sum_{\bi_{(n)}\in\scrI_{D,1,n}} A_{i_1} A_{i_2} \cdots A_{i_n}, \\
&U^{(n)} = \sum_{\bi_{(n)}\in\scrI_{U,1,n}} A_{i_1} A_{i_2} \cdots A_{i_n}. 
\end{align*}
Under (a2), $Q_{0,0}^{(n)}$, $D^{(n)}$ and $U^{(n)}$ are finite for every $n\ge 1$. Define $N$, $R$ and $G$ as 
\begin{align*}
&N = \sum_{n=0}^\infty Q_{0,0}^{(n)},\quad 
G = \sum_{n=1}^\infty D^{(n)}, \quad 
R = \sum_{n=1}^\infty U^{(n)},
\end{align*}
where $Q_{0,0}^{(0)}=I$. The following properties hold. 
\begin{lemma} \label{le:RandGmatrix_equations}
Assume (a1) and (a2). Then, $N$, $G$ and $R$ satisfy the following equations, including the case where both the sides of the equations diverge. 
\begin{align}
&R =  A_1 N, \\
&G = N  A_{-1}, \\
&R = R^2  A_{-1}+R  A_0+ A_1, \label{eq:Rmatrix_equation0} \\
&G =  A_{-1}+ A_0 G+ A_1 G^2, \label{eq:Gmatrix_equation0} \\
&N = I+A_0 N+A_1 G N = I+N A_0+N A_1 G. \label{eq:NandH_relation}
\end{align}
\end{lemma}

To make this paper self-contained, we give a proof of the lemma in Appendix \ref{sec:proof_RandGmatrix_equations}. 
From (\ref{eq:NandH_relation}), $N\ge I$ and $N\ge A_0 N+A_1 G N \ge A_0 +A_1 G$. Hence, if $N$ is finite, then $A_0 +A_1 G$ is also finite and we obtain 
\begin{equation}
(I-A_0-A_1 G)N=N(I-A_0-A_1 G)=I. 
\label{eq:NandH_relation2}
\end{equation}
We will use equation (\ref{eq:NandH_relation}) in this form. Here, we should note that much attention must be paid to matrix manipulation since the dimension of matrices are countably infinite, e.g., see Appendix A of \cite{Takahashi01}. 

For $\theta\in\mathbb{R}$, define a matrix function $A_*(\theta)$ as 
\[
A_*(\theta) = e^{-\theta} A_{-1} + A_0 + e^\theta A_1, 
\]
where $A_*=A_*(0)$. This $A_*(\theta)$ corresponds to a Feynman-Kac operator when the triplet $\{A_{-1},A_0,A_1\}$ is a Markov additive kernel (see, e.g., \cite{Ney87}). 
For $R$ and $G$, we have the following identity corresponding to the RG decomposition for a Markov additive process, which is also called a Winer-Hopf factorization, see identity (5.5) of \cite{Miyazawa11} and references therein. 
\begin{lemma} \label{le:As_WHfac}
Assume (a1) and (a2). If $R$, $G$ and $N$ are finite, we have, for $\theta\in\mathbb{R}$, 
\begin{align}
I- A_*(\theta) &= (I-e^\theta R) (I-H) (I-e^{-\theta} G), 
\label{eq:As_WHfact_RG}
\end{align}
where $H = A_0+ A_1 G = A_0+ A_1 N  A_{-1}$. 
\end{lemma}

\begin{proof}
Using identities (\ref{eq:NandH_relation2}), we obtain
\begin{align}
I- A_*(\theta) 
&= I-e^{-\theta} (I-H)N A_{-1} - A_0-e^\theta A_1 \cr
&= (I-A_0-e^\theta A_1 - A_1 N A_{-1})(I-e^{-\theta} G) \cr
&= (I-A_0-e^\theta A_1 N(I-H) - A_1 N A_{-1})(I-e^{-\theta} G) \cr
&= \bigl( (I-e^\theta R) (I-H) \bigr) (I-e^{-\theta} G),
\label{eq:As_WHfact_RG2}
\end{align}
where we use the fact that, by Lemma \ref{le:RandGmatrix_equations}, every term appeared in the expression such as $H N A_{-1} = A_0 G+A_1 G^2$ and $A_1 N A_{-1}=A_1 G$ are finite. Analogously, we have 
\begin{align}
I- A_*(\theta) 
&= I-e^{-\theta} A_{-1} - A_0-e^\theta A_1 N(I-H) \cr
&= (I-e^\theta R)(I-A_0-e^{-\theta} A_{-1}-A_1 N A_{-1}) \cr
&= (I-e^\theta R)(I-A_0-e^{-\theta} (I-H)N A_{-1}-A_1 N A_{-1})) \cr
&= (I-e^\theta R) \bigl( (I-H) (I-e^{-\theta} G) \bigr).
\label{eq:As_WHfact_RG3}
\end{align}
As a result, we obtain (\ref{eq:As_WHfact_RG}).
\end{proof}

Consider the following matrix quadratic equations of $X$:
\begin{align}
& X = X^2 A_{-1}+X A_0+A_1, \label{eq:Rmatrix_equation1} \\
& X = A_{-1}+A_0 X+A_1 X^2. \label{eq:Gmatrix_equation1}
\end{align}
By Lemma \ref{le:RandGmatrix_equations}, $R$ and $G$ are solutions to equations (\ref{eq:Rmatrix_equation1}) and (\ref{eq:Gmatrix_equation1}), respectively. 
Consider the following sequences of matrices:
\begin{align}
X_0^{(1)}=O,\quad X_n^{(1)} = \bigl(X_{n-1}^{(1)}\bigr)^2 A_{-1}+X_{n-1}^{(1)} A_0+A_1,\ n\ge 1, 
\label{eq:Xn1_iteration} \\
X_0^{(2)}=O,\quad X_n^{(2)} = A_{-1}+A_0 X_{n-1}^{(2)}+A_1 \bigl(X_{n-1}^{(2)}\bigr)^2,\ n\ge 1. 
\label{eq:Xn2_iteration}
\end{align} 
Like the case of usual QBD process, we can demonstrate that both the sequences $\{X_n^{(1)}\}_{n\ge 0}$ and $\{X_n^{(2)}\}_{n\ge 0}$ are nondecreasing and that if a nonnegative solution $X^*$ to equation (\ref{eq:Rmatrix_equation1}) (resp.\ equation (\ref{eq:Gmatrix_equation1})) exists, then for any $n\ge 0$, $X^*\ge X_n^{(1)}$ (resp.\ $X^*\ge X_n^{(2)}$). 
Furthermore, letting $R_n$ and $G_n$ be defined as 
\[
R_n = \sum_{k=1}^n U^{(k)},\quad 
G_n = \sum_{k=1}^n D^{(k)},
\]
we can also demonstrate that, for any $n\ge 1$, $R_n\le X_n^{(1)}$ and $G_n\le X_n^{(2)}$ hold. Hence, we immediately obtain the following facts. 
\begin{lemma} \label{le:RandGmatrix_solutions}
Assume (a1) and (a2). Then, $R$ and $G$ are the minimum nonnegative solutions to equations (\ref{eq:Rmatrix_equation1}) and (\ref{eq:Gmatrix_equation1}), respectively. 
Furthermore, we have $R=\lim_{n\to\infty} X_n^{(1)}$ and $G=\lim_{n\to\infty} X_n^{(2)}$. 
\end{lemma}

%
If $A_*$ is irreducible, $A_*(\theta)$ is also irreducible for any $\theta\in\mathbb{R}$. We, therefore, give the following condition. 
\begin{condition}
\begin{itemize}
\item[(a3)] $A_*$ is irreducible.
\end{itemize}
\end{condition}
Let $\chi(\theta)$ be the reciprocal of the convergence parameter of $A_*(\theta)$, i.e., $\chi(\theta)=\cp(A_*(\theta))^{-1}$. 
We say that a positive function $f(x)$ is log-convex in $x$ if $\log f(x)$ is convex in $x$. A log-convex function is also a convex function. Since every element of $A_*(\theta)$ is log-convex in $\theta$, we see, by Lemma \ref{le:cp_convex} in Appendix \ref{sec:cp_convex}, that $\chi(\theta)$ satisfies the following property. 
\begin{lemma} \label{le:chi_convex}
Under (a1) through (a3), $\chi(\theta)$ is log-convex in $\theta\in\mathbb{R}$. 
\end{lemma}

Let $\gamma^\dagger$ be the infimum of $\chi(\theta)$, i.e., 
\[
\gamma^\dagger = \inf_{\theta\in\mathbb{R}} \chi(\theta) = \Big( \sup_{\theta\in\mathbb{R}} \cp(A_*(\theta)) \Big)^{-1},  
\]
and define a set $\bar{\Gamma}$ as 
\[
\bar{\Gamma} = \{\theta\in\mathbb{R}; \chi(\theta)\le 1\} = \{ \theta\in\mathbb{R}; \cp(A_*(\theta)) \ge 1 \}. 
\]
By Lemma \ref{le:chi_convex}, if $\gamma^\dagger<1$ and $\bar{\Gamma}$ is bounded, then $\bar{\Gamma}$ is a line segment and there exist just two real solutions to equation $\chi(\theta)=\cp(A_*(\theta))^{-1}=1$. We denote the solutions by $\underline{\theta}$ and $\bar{\theta}$, where $\underline{\theta}<\bar{\theta}$. 
When $\gamma^\dagger=1$, we define $\underline{\theta}$ and $\bar{\theta}$ as $\underline{\theta}=\min\{\theta\in\mathbb{Z}; \chi(\theta)=1\}$ and $\bar{\theta}=\max\{\theta\in\mathbb{Z}; \chi(\theta)=1\}$, respectively. It is expected that $\underline{\theta}=\bar{\theta}$ if $\gamma^\dagger=1$, but it is not obvious. 
If $\gamma^\dagger\le1$ and $\bar{\Gamma}$ is bounded, there exists a $\theta\in\bar{\Gamma}$ such that $\gamma^\dagger=\chi(\theta)$. 
We give the following condition. 
\begin{condition}
\begin{itemize}
\item[(a4)] $\bar{\Gamma}$ is bounded. 
\end{itemize}
\end{condition}

If $A_{-1}$ (resp.\ $A_1$) is a zero matrix, every element of $A_*(\theta)$ is monotone increasing (resp.\ decreasing) in $\theta$ and $\bar{\Gamma}$ is unbounded. Hence, if $\gamma^\dagger\le 1$, condition (a4) implies (a1).
The following properties correspond to those in Lemma 2.3 of \cite{Kijima93}. 
\begin{proposition} \label{pr:Gmatrix_existence} 
Assume (a2) through (a4). 
\begin{itemize} 
\item[(i)] If $\gamma^\dagger\le 1$, then $R$ and $G$ are finite. 
\item[(ii)] If $R$ is finite and there exist a $\theta_0\in\mathbb{R}$ and nonnegative nonzero vector $\bu$ such that $e^{\theta_0} \bu^\top R=\bu^\top$, then $\gamma^\dagger\le 1$. 
\item[(ii')] If $G$ is finite and there exist a $\theta_0\in\mathbb{R}$ and nonnegative nonzero vector $\bv$ such that $e^{\theta_0}  G \bv=\bv$, then $\gamma^\dagger\le 1$. 
\end{itemize}
\end{proposition}
\begin{proof}
{\it Statement (i).}\quad Assume $\gamma^\dagger\le 1$ and let $\theta^\dagger$ be a real number satisfying $\chi(\theta^\dagger)=\gamma^\dagger$. Since $A_*(\theta^\dagger)$ is irreducible, by Lemma 1 and Theorem 1 of \cite{Pruitt64}, there exists a positive vector $\bu$ satisfying $(\gamma^\dagger)^{-1} \bu^\top A_*(\theta^\dagger)  \le \bu^\top$. 
For this $\bu$, we obtain, by induction using (\ref{eq:Xn1_iteration}), inequality $e^{\theta^\dagger} \bu^\top X_n^{(1)} \le  \bu^\top$ for any $n\ge 0$. 
Hence, the sequence $\{X_n^{(1)}\}$ is element-wise nondecreasing and bounded, and the limit of the sequence, which is the minimum nonnegative solution to equation (\ref{eq:Rmatrix_equation1}), exists. Existence of the minimum nonnegative solution to equation (\ref{eq:Gmatrix_equation1}) is analogously proved. As a result, by Lemma \ref{le:RandGmatrix_solutions}, both $R$ and $G$ are finite.

{\it Statements (ii) and (ii').}\quad 
Assume the condition of Statement (ii). Then, we have 
\begin{align}
\bu^\top 
&=  e^{\theta_0} \bu^\top R
=  e^{\theta_0} \bu^\top ( R^2 A_{-1}+R A_0+A_1 ) 
= \bu^\top A_*(\theta_0), 
\label{eq:z0uR}
\end{align}
and this leads us to $\gamma^\dagger\le \chi(\theta_0)=\cp(A_*(\theta_0))^{-1}\le 1$. 
Statement (ii') can analogously be proved. 
\end{proof}

\begin{remark}
In statement (ii) of Proposition \ref{pr:Gmatrix_existence}, if such $\theta_0$ and $\bu$ exist, then, by (\ref{eq:z0uR}) and irreducibility of $A_*(\theta_0)$, we have $\theta_0=\theta^\dagger$ and $\bu$ is positive. An analogous result also holds for statement (ii'). 
\end{remark}

\begin{remark} \label{rm:chi_unbounded}
Consider the following nonnegative matrix $P$:
\[
P 
= \begin{pmatrix}
 \ddots & \ddots & \ddots & & & \cr
& A_{-1} & A_0 & A_1 & & \cr
& & A_{-1} & A_0 & A_1 & \cr
& & & \ddots & \ddots & \ddots  
\end{pmatrix}.  
\]
If the triplet $\{A_{-1},A_0,A_1\}$ is a Markov additive kernel, this $P$ corresponds to the transition probability matrix of a Markov additive process governed by the triplet. By Proposition \ref{pr:chi_unbounded} in Appendix \ref{sec:chi_unbounded}, if $P$ is irreducible, then $\chi(\theta)$ is unbounded in both the directions of $\theta$ and $\bar{\Gamma}$ is bounded. 
\end{remark}

%
For the convergence parameters of $R$ and $G$, the following properties holds (for the case where the dimension of $A_*$ is finite, see Lemma 2.2 of \cite{He09} and Lemma 2.3 of  \cite{Miyazawa09}). 
\begin{lemma} \label{le:RandG_cp}
Assume (a2) through (a4).  If $\gamma^\dagger\le 1$ and $N$ is finite, then we have 
\begin{equation}
\cp(R) = e^{\bar{\theta}},\quad 
\cp(G) = e^{-\underline{\theta}}. 
\end{equation}
\end{lemma}
\begin{proof}
Since $\gamma^\dagger\le 1$ and $\bar{\Gamma}$ is bounded, $\bar{\theta}$ and $\underline{\theta}$ exist and they are finite. Furthermore, $R$ and $G$ are finite. 
For a $\theta\in\mathbb{R}$ such that $\chi(\theta)\le 1$, let $\bu$ be a positive vector satisfying $\bu^\top A_*(\theta)\le \bu^\top$. Such a $\bu$ exists since $A_*(\theta)$ is irreducible. As mentioned in the proof of Proposition \ref{pr:Gmatrix_existence}, for $X^{(1)}_n$ defined by (\ref{eq:Xn1_iteration}), if $\chi(\theta)\le 1$, then we have $e^\theta \bu^\top X^{(1)}_n\le \bu^\top$ for any $n\ge 0$ and this implies $e^\theta \bu^\top R\le \bu^\top$. 
Analogously, if $\chi(\theta)\le 1$, then there exists a positive vector $\bv$ satisfying $A_*(\theta)\bv\le \bv$ and we have $e^{-\theta} G \bv\le \bv$.
Therefore, setting $\theta$ at $\bar{\theta}$, we obtain $e^{\bar{\theta}} \bu^\top R\le \bu^\top$, and setting $\theta$ at $\underline{\theta}$, we obtain $e^{-\underline{\theta}} G \bv\le \bv$. Since $\bu$ and $\bv$ are positive, this leads us to $\cp(R)\ge e^{\bar{\theta}}$ and $\cp(G)\ge e^{-\underline{\theta}}$. 

Next, in order to prove $\cp(R)\le e^{\bar{\theta}}$, we apply a technique similar to that used in the proof of Theorem 1 of \cite{Pruitt64}. Suppose $\cp(R)>e^{\bar{\theta}}$. Then, there exists an $\varepsilon>0$ such that 
\[
\tilde{R}(\bar{\theta}+\varepsilon)=\sum_{n=0}^\infty e^{(\bar{\theta}+\varepsilon) n} R^n < \infty. 
\]
This $\tilde{R}(\bar{\theta}+\varepsilon)$ satisfies $e^{\bar{\theta}+\varepsilon} R \tilde{R}(\bar{\theta}+\varepsilon) = \tilde{R}(\bar{\theta}+\varepsilon)-I \le \tilde{R}(\bar{\theta}+\varepsilon)$. Hence, for $j\in\mathbb{Z}_+$, letting $\bv_j$ is the $j$-th column vector of $\tilde{R}(\bar{\theta}+\varepsilon)$, we have $e^{\bar{\theta}+\varepsilon} R \bv_j \le \bv_j$.
Furthermore, we have $e^{\bar{\theta}+\varepsilon} R \tilde{R}(\bar{\theta}+\varepsilon) \ge  e^{\bar{\theta}+\varepsilon} R \ge e^{\bar{\theta}+\varepsilon} X_1^{(1)} = e^{\bar{\theta}+\varepsilon} A_1$, and condition (a4) implies $A_1$ is nonzero. Hence, for some $j\in\mathbb{Z}_+$, both $R\bv_j$ and  $\bv_j$ are nonzero. Set $\bv$ at such a vector $\bv_j$. 
We have $\cp(G)\ge e^{-\underline{\theta}}> e^{-(\bar{\theta}+\varepsilon)}$. Hence, using (\ref{eq:NandH_relation2}) and (\ref{eq:As_WHfact_RG}), we obtain  
\begin{equation}
(I-A_*(\bar{\theta}+\varepsilon)) (I-e^{-(\bar{\theta}+\varepsilon)} G)^{-1} N \bv = (I-e^{\bar{\theta}+\varepsilon} R)\bv \ge \bzero, 
\end{equation}
where $\by=(I-e^{-(\bar{\theta}+\varepsilon)} G)^{-1} N \bv = \sum_{n=0}^\infty e^{-n (\bar{\theta}+\varepsilon)} G^n\,N \bv\ge N \bv$. 
Suppose $N \bv=\bzero$, then we have $R \bv = A_1 N \bv = \bzero$ and this contradicts that $R \bv$ is nonzero. Hence, $N \bv$ is nonzero and $\by$ is also nonzero and nonnegative. 
Since $A_*(\bar{\theta}+\varepsilon)$ is irreducible, the inequality $A_*(\bar{\theta}+\varepsilon) \by \le \by$ implies that $\by$ is positive and $\cp(A_*(\bar{\theta}+\varepsilon))\ge 1$. This contradicts that $\cp(A_*(\bar{\theta}+\varepsilon))=\chi(\bar{\theta}+\varepsilon)^{-1}<\chi(\bar{\theta})^{-1}=1$, and we obtain $\cp(R)\le e^{\bar{\theta}}$. 
In a similar manner, we can also obtain $\cp(G)\le e^{-\underline{\theta}}$, and this completes the proof.
\end{proof}


Lemma \ref{le:RandG_cp} requires that $N$ is finite, but it cannot easily be verified since finiteness of $R$ and $G$ does not always imply that of $N$. We, therefore, introduce the following condition. 
\begin{condition}
\begin{itemize}
\item[(a5)] The nonnegative matrix $Q$ is irreducible. 
\end{itemize}
\end{condition}

Condition (a5) implies (a1), (a3) and (a4), i.e., under conditions (a2) and (a5), $A_{-1}$ and $A_1$ are nonzero, $A_*$ is irreducible and $\bar{\Gamma}$ is bounded. 
Let $\tilde{Q}$ be the fundamental matrix of $ Q$, i.e., $\tilde{Q}=\sum_{n=0}^\infty Q^n$. For $n\ge 0$, $Q_{0,0}^{(n)}$ is the $(0,0)$-block of $Q^n$, and $N$ is that of $\tilde{Q}$. Hence, we see that all the elements of $N$ simultaneously converge or diverge, finiteness of $R$ or $G$ implies that of $N$ and if $N$ is finite, it is positive. 
Furthermore, under (a2) and (a5), since $R$ is given as $R=A_1 N$ and $N$ is positive, each row of $R$ is zero or positive and we obtain the following proposition, which asserts that $R$ behaves just like an irreducible matrix. 
\begin{proposition} \label{pr:R_u}
Assume (a2) and (a5). If $R$ is finite, then it always satisfies one of the following two statements.
\begin{itemize}
\item[(i)] There exists a positive vector $\bu$ such that $e^{\bar{\theta}} \bu^\top R = \bu^\top$. 
\item[(ii)] $\sum_{n=0}^\infty e^{\bar{\theta}n} R^n <\infty$. 
\end{itemize}
\end{proposition}

Since the proof of this proposition is elementary and lengthy,  we put it in Appendix \ref{sec:proof_R_u}. By applying the same technique as that used in the proof of Theorem 4.1 of \cite{Kobayashi10}, we also obtain the following result.
\begin{corollary} \label{co:cpRlimit}
Assume (a2) and (a5). For $i,j\in\mathbb{Z}_+$, if every element in the $i$-th row of $A_1$ is zero, we have $[R^n]_{i,j}=0$ for all $n\ge 1$; otherwise, we have $[R^n]_{i,j}>0$ for all $n\ge 1$ and 
\begin{equation}
\lim_{n\to\infty} ([R^n]_{i,j})^{\frac{1}{n}}=e^{-\bar{\theta}}. 
\label{eq:cpRlimit}
\end{equation}

\end{corollary}

To make this paper self-contained, we give a proof of the corollary in Appendix  \ref{sec:proof_R_u}.
By Theorem 2 of \cite{Pruitt64}, if the number of nonzero elements of each row of $A_*$ is finite, there exists a positive vector $\bu$ satisfying $\bu^\top A_*(\bar{\theta})=\bu^\top$. Also, if the number of nonzero elements of each column of $A_*$ is finite, there exists a positive vector $\bv$ satisfying $A_*(\underline{\theta})\bv=\bv$. To use this property, we give the following condition.
\begin{condition}
\begin{itemize}
\item[(a6)] The number of positive elements of each row and column of $A_*$ is finite.
\end{itemize}
\end{condition}
It is obvious that (a6) implies (a2). Under (a6), we can refine Proposition \ref{pr:Gmatrix_existence}, as follows.

\begin{proposition} \label{pr:RandG_existence2}
Assume (a5) and (a6). Then, $\gamma^\dagger\le 1$ if and only if $R$ and $G$ are finite.
\end{proposition}
\begin{proof}
By Proposition \ref{pr:Gmatrix_existence}, if $\gamma^\dagger\le 1$, then both $R$ and $G$ are finite. We, therefore, prove the converse.
Assume that $R$ and $G$ are finite. Then, $N$ is also finite and, by Lemma \ref{le:RandG_cp}, we have $\cp(R)=e^{\bar{\theta}}$. 
First, consider case (i) of Proposition \ref{pr:R_u} and assume that there exists a positive vector $\bu$ such that $e^{\bar{\theta}} \bu R = \bu$. Then, by statement (ii) of Proposition \ref{pr:Gmatrix_existence}, we have $\gamma^\dagger \le 1$.  
Next, consider case (ii) of Proposition \ref{pr:R_u} and assume $\sum_{n=0}^\infty e^{n \bar{\theta}} R^n <\infty$. Then, we have $(I-e^{\underline{\theta}} R)^{-1}=\sum_{n=0}^\infty e^{n \underline{\theta}} R^n<\infty$ since $\underline{\theta}\le \bar{\theta}$. Hence, we obtain, from (\ref{eq:NandH_relation2}) and (\ref{eq:As_WHfact_RG}) ,  
\begin{align}
 N (I-e^{\underline{\theta}} R)^{-1} (I-A_*(\underline{\theta})) = (I-e^{-\underline{\theta}} G). \label{eq:AsGandR_relation2}
\end{align}
Under the assumption of the proposition, there exists a positive vector $\bv$ satisfying $A_*(\underline{\theta})\bv=\bv$ since $\cp(A_*(\underline{\theta}))=1$. Hence, from (\ref{eq:AsGandR_relation2}), we obtain, for this $\bv$, $e^{-\underline{\theta}} G\bv =\bv$, and by statement (ii') of Proposition \ref{pr:Gmatrix_existence}, we have $\gamma^\dagger \le 1$. 
This completes the proof. 
\end{proof}

%
Recall that $\gamma^\dagger$ is defined as $\gamma^\dagger=\inf_{\theta\in\mathbb{R}}\cp(A_*(\theta))^{-1}$. Since if $Q$ is irreducible, all the elements of $\tilde{Q}=\sum_{n=0}^\infty Q^n$ simultaneously converge or diverge, we obtain, from Proposition \ref{pr:RandG_existence2}, the following result. 
\begin{proposition} \label{pr:Q_cp}
Assume (a5) and (a6). Then, $\gamma^\dagger\le 1$ if and only if $\tilde{Q}$ is finite. 
\end{proposition}
\begin{proof}
Under the assumption of the proposition, if $\gamma^\dagger \le 1$, then, by Proposition \ref{pr:RandG_existence2}, $R$ and $G$ are finite. Since $Q$ is irreducible, this implies that $N$ is finite and $\tilde{Q}$ is also finite. 
On the other hand, if $\tilde{Q}$ is finite, then $N$ is finite and $R$ and $G$ are also finite since the number of positive elements of each row of $A_1$ and that of each column of $A_{-1}$ are finite. Hence, by Proposition \ref{pr:RandG_existence2}, $\gamma^\dagger\le 1$ and this completes the proof. 
\end{proof}

By this proposition, we obtain the main result of this section, as follows. 
\begin{lemma} \label{le:Q_cp}
Under (a5) and (a6), we have 
\begin{equation}
\cp(Q) = (\gamma^\dagger)^{-1} = \sup_{\theta\in\mathbb{R}} \cp(A_*(\theta))
\end{equation}
and $Q$ is $(\gamma^\dagger)^{-1}$-transient. 
\end{lemma}
\begin{proof}
For $\beta>0$, $\beta Q$ is a nonnegative block tridiagonal matrix, whose block matrices are given by $\beta A_{-1}$, $\beta A_0$ and $\beta A_1$. Hence, the assumption of this lemma also holds for $\beta Q$. 
Define $\gamma(\beta)$ as 
\begin{equation}
\gamma(\beta) 
= \inf_{\theta\in\mathbb{R}} \cp(\beta A_*(\theta))^{-1} 
= \beta \inf_{\theta\in\mathbb{R}} \cp(A_*(\theta))^{-1}
= \beta \gamma^\dagger.
\end{equation}
By Proposition \ref{pr:Q_cp}, if $\gamma(\beta) = \beta \gamma^\dagger\le 1$, then the fundamental matrix of $\beta Q$, $\widetilde{\beta Q}$, is finite and $\cp(\beta Q)=\beta^{-1} \cp(Q)\ge 1$. Hence, if $\beta\le (\gamma^\dagger)^{-1}$, then $\cp(Q)\ge \beta$. Setting $\beta$ at $(\gamma^\dagger)^{-1}$, we obtain $\cp(Q)\ge (\gamma^\dagger)^{-1}$. 
Next we prove $\cp(Q)\le (\gamma^\dagger)^{-1}$. Suppose $\cp(Q)> (\gamma^\dagger)^{-1}$, then there exists an $\varepsilon>0$ such that the fundamental matrix of $((\gamma^\dagger)^{-1}+\varepsilon) Q$ is finite. By Proposition \ref{pr:Q_cp}, this implies 
\begin{equation}
\gamma((\gamma^\dagger)^{-1}+\varepsilon)) = ((\gamma^\dagger)^{-1}+\varepsilon) \gamma^\dagger = 1+\varepsilon \gamma^\dagger \le 1, 
\end{equation}
and we obtain $\gamma^\dagger\le 0$. This contradicts $\gamma^\dagger> 0$, which is obtained from the irreducibility of $A_*$. Hence, we obtain $\cp(Q)\le (\gamma^\dagger)^{-1}$.
Setting $\beta$ at $(\gamma^\dagger)^{-1}$, we have $\gamma(\beta)=\gamma((\gamma^\dagger)^{-1})\le 1$ and, by Proposition \ref{pr:Q_cp}, the fundamental matrix of $\beta Q=(\gamma^\dagger)^{-1}Q$ is finite. This means $Q$ is $(\gamma^\dagger)^{-1}$-transient.
\end{proof}

\begin{remark}
In the case where the phase space is finite, Lemma \ref{le:Q_cp} corresponds to Lemma 2.3 of \cite{Miyazawa15}. Assuming condition (a6), we extended that lemma to the case of infinite phase space. 
\end{remark}

%
\begin{remark} \label{re:multidiagonal_cp}
For nonnegative block multidiagonal matrices, a property similar to Lemma \ref{le:Q_cp} holds. We demonstrate it in the case of block quintuple-diagonal matrix.
Let $Q$ be a nonnegative block matrix defined as 
\[
Q 
= \begin{pmatrix}
A_0 & A_1 & A_2 & & & & \cr
A_{-1} & A_0 & A_1 & A_2 & & & \cr
A_{-2} & A_{-1} & A_0 & A_1 & A_2& & \cr
& A_{-2} & A_{-1} & A_0 & A_1 & A_2& \cr
& & \ddots & \ddots & \ddots & \ddots & \ddots 
\end{pmatrix}, 
\]
where $A_i,\,i\in\{-2,-1,0,1,2\}$, are nonnegative square matrices with a countable dimension.
For $\theta\in\mathbb{R}$, define a matrix function $A_*(\theta)$ as 
\begin{equation}
A_*(\theta) = \sum_{i=-2}^2 e^{i \theta} A_i.
\end{equation}
Then, assuming that $Q$ is irreducible and the number of positive elements of each row and column of $A_*(0)$ is finite, we can obtain 
\begin{equation}
\cp(Q) = \sup_{\theta\in\mathbb{R}} \cp(A_*(\theta)).
\label{eq:cpQ_5diagonal}
\end{equation}
Here we prove this equation. Define blocks $\hat{A}_i,\,i\in\{-1,0,1\}$, as
\[
\hat{A}_{-1} = 
\begin{pmatrix} A_{-2} & A_{-1} \cr O & A_{-2} \end{pmatrix}, \quad
\hat{A}_0 = 
\begin{pmatrix} A_0 & A_1 \cr A_{-1} & A_0 \end{pmatrix}, \quad
\hat{A}_1 = 
\begin{pmatrix} A_2 & O \cr A_1 & A_2 \end{pmatrix},  
\]
then $Q$ is represented in block tridiagonal form by using these blocks. For $\theta\in\mathbb{R}$, define a matrix function $\hat{A}_*(\theta)$ as 
\[
\hat{A}_*(\theta) 
= e^{-\theta} \hat{A}_{-1} + \hat{A}_0 + e^{\theta} \hat{A}_1
= \begin{pmatrix} 
e^{-\theta} A_{-2} + A_0 + e^{\theta} A_2 & e^{-\theta/2} (e^{-\theta/2} A_{-1} + e^{\theta/2} A_1) \cr 
e^{\theta/2} (e^{-\theta/2} A_{-1} + e^{\theta/2} A_1) & e^{-\theta} A_{-2} + A_0 + e^{\theta} A_2
\end{pmatrix}, 
\]
then, by Lemma \ref{le:Q_cp}, we have $\cp(Q) = \sup_{\theta\in\mathbb{R}} \cp(\hat{A}_*(\theta))$. To prove equation (\ref{eq:cpQ_5diagonal}), it, therefore, suffices to show that, for any $\theta\in\mathbb{R}$, 
\begin{align}
\cp(A_*(\theta/2)) 
&= \sup\{ \alpha\in\mathbb{R}_+; \alpha \bx^\top A_*(\theta/2)\le \bx^\top\ \mbox{for some}\ \bx>\bzero \} \cr
&= \sup\{ \alpha\in\mathbb{R}_+; \alpha \hat{\bx}^\top \hat{A}_*(\theta)\le \hat{\bx}^\top\ \mbox{for some}\ \hat{\bx}>\bzero \} 
= \cp(\hat{A}_*(\theta)). 
\label{eq:AshatAs}
\end{align}
For $\theta\in\mathbb{R}$ and $\alpha\in\mathbb{R}_+$, if $\alpha \bx^\top A_*(\theta/2)\le \bx^\top$ for some $\bx>\bzero$, then, letting $\hat{\bx}^\top=(\bx^\top, e^{-\theta/2} \bx^\top)$, we have $\alpha \hat{\bx}^\top \hat{A}_*(\theta) \le \hat{\bx}^\top$. 
On the other hand, if $\alpha \hat{\bx}^\top \hat{A}_*(\theta) \le \hat{\bx}^\top$ for some $\hat{\bx}^\top=(\hat{\bx}_1^\top,\hat{\bx}_2^\top)>\bzero^\top$, then letting $\bx=\hat{\bx}_1+e^{\theta/2} \hat{\bx}_2$, we have $\alpha \bx^\top A_*(\theta/2)\le \bx^\top$. 
As a result, we obtain equations (\ref{eq:AshatAs}). 

\end{remark}

%
%
\section{Markov modulated random walks: preliminaries} \label{sec:model}

We give some assumptions and propositions for the $d$-dimensional MMRW $\{\bY_n\}=\{(\bX_n,J_n)\}$ defined in Sect.\ \ref{sec:intro}. First, we assume the following condition. 
\begin{assumption} \label{as:P_irreducible}
The $d$-dimensional MMRW $\{\bY_n\}$ is irreducible. 
\end{assumption}

Under this assumption, for any $\btheta\in\mathbb{R}^d$, $A_*(\btheta)$ is also irreducible. 
Denote $A_*(\bzero)$ by $A_*$, which is the transition probability matrix of the background process $\{J_n\}$. In order to use the results in the previous section, we assume the following condition.
\begin{assumption} \label{as:Ass_boundary}
The number of positive elements in every low and column of $A_*$ is finite. 
\end{assumption}

Define the mean increment vector $\ba=(a_1,a_2,...,a_d)$ as 
\[
a_i = \lim_{n\to\infty} \frac{1}{n} \sum_{k=1}^n (X_{i,k}-X_{i,k-1}), \quad i=1,2,...,d.
\]
We assume these limits exist with probability one. 
With respect to the occupation measures defined in Sect.\ \ref{sec:intro}, the following property holds.
\begin{proposition} \label{pr:finiteness_tildeQ}
If there exists some $i\in\{1,2,...,d\}$ such that $a_i<0$, then, for any $\by\in\mathbb{S}_+$, the occupation measure $(\tilde{q}_{\by,\by'}; \by'\in\mathbb{S}_+)$ is finite, i.e., 
\begin{equation}
\sum_{\by'\in\mathbb{S}_+} \tilde{q}_{\by,\by'} = \mathbb{E}(\tau\,|\,\bY_0=\by) < \infty, 
\end{equation}
where $\tau$ is the stopping time at which $\{\bY_n\}$ enters $\mathbb{S}\setminus\mathbb{S}_+$ for the first time. 
\end{proposition}
\begin{proof}
Without loss of generality, we assume $a_1<0$. Let $\check{\tau}$ be the stopping time at which $X_{1,n}$ becomes less than 0 for the first time, i.e., $\check{\tau} = \inf\{n\ge 0; X_{1,n}<0 \}$. Since $\{(x_1,x_2,...,x_d,j)\in\mathbb{S}; x_1<0 \} \subset \mathbb{S}\setminus\mathbb{S}_+$, we have $\tau\le \check{\tau}$, and this implies that, for any $\by\in\mathbb{S}_+$, 
\begin{equation}
\mathbb{E}(\tau\,|\,\bY_0=\by)\le\mathbb{E}(\check{\tau}\,|\,\bY_0=\by). 
\end{equation}
Next, we demonstrate that $\mathbb{E}(\check{\tau}\,|\,\bY_0=\by)$ is finite. For $i\in\{-1,0,1\}$, define a matrix $\check{A}_i$ as 
\[
\check{A}_i = \sum_{(i_2,i_3,...,i_d)\in\{-1,0,1\}^{d-1}} A_{(i,i_2,i_3,...,i_d)},  
\]
and consider a one-dimensional QBD process $\{\check{\bY}_n\}=\{(\check{X}_n,\check{J}_n)\}$ on $\mathbb{Z}_+\times S_0$, having $\check{A}_{-1}$, $\check{A}_0$ and $\check{A}_1$ as transition probability blocks when $\check{X}_n>0$. We assume the transition probability blocks that governs transitions of the QBD process when $\check{X}_n=0$ are given appropriately. 
Define a stopping time $\check{\tau}^Q$ as $\check{\tau}^Q=\inf\{n\ge 0; \check{X}_n=0\}$. Since $a_1$ is the mean increment of the QBD process when $\check{X}_n>0$,  the assumption of $a_1<0$ implies that, for any $(x,j)\in\mathbb{Z}_+\times S_0$, $\mathbb{E}(\check{\tau}^Q\,|\,\check{Y}_0=(x,j)) < \infty$. 
We, therefore, have for any $\by=(x_1,x_2,...,x_d,j)\in\mathbb{S}_+$, 
\begin{equation}
\mathbb{E}(\check{\tau}\,|\,\bY_0=\by) = \mathbb{E}(\check{\tau}^Q\,|\,\check{Y}_0=(x_1+1,j)) < \infty,
\end{equation}
and this completes the proof.
\end{proof}

Hereafter, we assume the following condition. 
\begin{assumption} \label{as:finiteness_tildeQ}
For some $i\in\{1,2,...,d\}$, $a_i<0$.
\end{assumption}

\begin{remark}
If $A_*$ is positive recurrent, the mean increment vector $\ba=(a_1,a_2,...,a_d)$ is given as 
\begin{equation}
a_i = \bpi_* \left(\frac{\partial}{\partial \theta_i} A_*(\btheta)\,\Big|_{\btheta=\bzero}\right) \bone,\ i=1,2,...,d,
\end{equation}
where $\bpi_*$ is the stationary distribution of $A_*$ and $\bone$ a column vector of $1$'s whose dimension is determined in context.
\end{remark}

%
We say that a positive function $f(\bx)$ is log-convex in $\bx\in\mathbb{R}^d$ if $\log f(\bx)$ is convex in $\bx$. A log-convex function is also a convex function. 
By Lemma \ref{le:cp_convex} in Appendix \ref{sec:cp_convex}, the following property holds. 
\begin{proposition} \label{pr:chiconvex}
$\cp(A_*(\btheta))^{-1}$ is log-convex and hence convex in $\btheta \in \mathbb{R}^d$. 
\end{proposition}

Let $\bar{\Gamma}$ be the closure of $\Gamma$, i.e., $\bar{\Gamma}=\{ \btheta\in\mathbb{R}^d;  \cp(A_*(\btheta))^{-1} \le 1\}$. By Proposition \ref{pr:chiconvex}, $\bar{\Gamma}$ is a convex set. By Remark \ref{rm:chi_unbounded} and Proposition \ref{pr:chi_unbounded} in Appendix \ref{sec:chi_unbounded}, the following property holds.
\begin{proposition} \label{pr:barGamma_bounded}
$\bar{\Gamma}$ is bounded. 
\end{proposition}

%
For $\by=(\bx,j)\in\mathbb{S}_+$, we give an asymptotic inequality for the occupation measure $(\tilde{q}_{\by,\by'}; \by'\in\mathbb{S}_+)$.  
Under Assumption \ref{as:finiteness_tildeQ}, the occupation measure is finite and $(\tilde{q}_{\by,\by'}/E(\tau\,|\,\bY_0=\by); \by'\in\mathbb{S}_+)$ becomes a probability measure. Let $\bY=(\bX,J)$ be a random variable subject to the probability measure, i.e., $P(\bY=\by')= \tilde{q}_{\by,\by'}/E(\tau\,|\,\bY_0=\by)$ for $\by'\in\mathbb{S}_+$.
By the Markov's inequality, for $\btheta\in\mathbb{R}^d$ and for $\bc\in\mathbb{R}_+^d$ such that $\bc\ne\bzero$, we have, for $j'\in S_0$, 
\begin{align*}
\mathbb{E}(e^{\langle \bX,\btheta \rangle} 1(J=j')) 
&\ge e^{k \langle \bc,\btheta \rangle} P(e^{\langle \bX,\btheta \rangle} 1(J=j') \ge e^{k \langle \bc,\btheta \rangle}) \cr
&= e^{k \langle \bc,\btheta \rangle} P(\langle \bX,\btheta \rangle \ge \langle k \bc,\btheta \rangle, J=j' ) \cr
&\ge e^{k \langle \bc,\btheta \rangle} P(\bX \ge k \bc, J=j' ).
\end{align*}
This implies that, for every $\bl\in\mathbb{Z}_+^d$, 
\begin{align}
& [\Phi_{\bx}(\btheta)]_{j,j'} 
\ge e^{k \langle \bc,\btheta \rangle} \sum_{\bx'\ge k \bc} \tilde{q}_{\by,(\bx',j')}
\ge e^{k \langle \bc,\btheta \rangle} \tilde{q}_{\by,(k (\lceil \bc \rceil +\bl),j')}, 
\end{align}
where $\lceil \bc \rceil=(\lceil c_1 \rceil,\lceil c_2 \rceil,...,\lceil c_d \rceil)$ and $\lceil x \rceil$ is the smallest integer greater than or equal to $x$. Hence, considering the convergence domain of $\Phi_{\bx}(\btheta)$, we immediately obtain the following basic inequality.
\begin{lemma} \label{le:limsup_tildeqnn}
For any $\bc\in\mathbb{Z}_+^d$ such that $\bc\ne\bzero$ and for every $(\bx,j)\in\mathbb{S}_+$, $j'\in S_0$ and $\bl\in\mathbb{Z}_+^d$,  
\begin{equation}
 \limsup_{k\to\infty} \frac{1}{k} \log \tilde{q}_{(\bx,j),(k \bc+\bl,j')} \le - \sup_{\btheta\in\calD_{\bx}} \langle \bc,\btheta \rangle. 
 \label{eq:tildeq_upper}
\end{equation}
\end{lemma}

%
%

%
%
\section{QBD representations for the MMRW} \label{sec:QBDrepresentation}

In this section, we make one-dimensional QBD processes with countably many phases from the $d$-dimensional MMRW defined in Sect.\ \ref{sec:intro} and obtain upper bounds for the convergence parameters of their rate matrices. Those upper bounds will give lower bounds for the asymptotic decay rates of the occupation measures in the original MMRW.

\subsection{QBD representation with level direction vector $\bone$}

Let $\{\bY_n\}=\{(\bX_n,J_n)\}$ be a $d$-dimensional MMRW.  
In order to use the results in Sect.\ \ref{sec:RandGmatrix}, hereafter, we assume the following condition.
\begin{assumption} \label{as:Q_irreducible}
$P_+$ is irreducible.
\end{assumption}

Under this assumption, $P$ is irreducible regardless of Assumption \ref{as:P_irreducible} and every element of $\tilde{P}_+$ is positive. 
Let $\tau$ be the stopping time defined in Sect.\ \ref{sec:intro}, i.e.,  $\tau=\inf\{n\ge 0; \bY_n\in\mathbb{S}\setminus\mathbb{S}_+\}$. According to Example 4.2 of \cite{Miyazawa12}, define a one-dimensional absorbing QBD process $\{\hat{\bY}_n\}=\{(\hat{X}_n,\hat{\bJ}_n)\}$ as
\[
\hat{X}_n = \min_{1\le i\le d} X_{i,\tau\wedge n},\quad 
\hat{\bJ}_n = (\hat{Z}_{0,n},\hat{Z}_{1,n},...,\hat{Z}_{d-1,n},\hat{J}_{n}),
\]
where $x\wedge y= \min\{x,y\}$, $\hat{Z}_{0,n}=\min\{i\in\{1,2,...,d\}; X_{i,\tau\wedge n}=\hat{X}_n\}$,   
\[
\hat{Z}_{i,n} = \left\{ \begin{array}{ll}
X_{i,\tau\wedge n}-\hat{X}_n, & i<\hat{Z}_{0,n}, \cr
X_{i+1,\tau\wedge n}-\hat{X}_n, & i\ge \hat{Z}_{0,n},  
\end{array} \right.
\quad i=1,2,...,d-1,
\]
and $\hat{J}_n=J_{\tau\wedge n}$. We restrict the state space of $\{\hat{\bY}_n\}$ to $\mathbb{Z}_+\times(\mathbb{N}_d\times\mathbb{Z}_+^{d-1}\times S_0)$, where $\mathbb{N}_d=\{1,2,...,d\}$. 
For $k\in\mathbb{Z}_+$, the $k$-th level set of $\{\hat{\bY}_n\}$ is given by $\mathbb{L}_k=\{ (\bx,j)=((x_1,x_2,...,x_d),j)\in\mathbb{S}_+; \min_{1\le i\le d} x_i = k \}$ and they satisfy, for $k\ge 0$,  
\begin{equation}
\mathbb{L}_{k+1} = \{(\bx+\bone,j); (\bx,j)\in\mathbb{L}_k \}.
\end{equation}
This means that $\{\hat{\bY}_n\}$ is a QBD process with level direction vector $\bone$. 
The transition probability matrix of $\{\hat{Y}_n\}$ is given in block tri-diagonal form as
\begin{equation} \label{eq:blockformP}
\hat{P} = 
\begin{pmatrix}
\hat{A}_0 & \hat{A}_1 & & & \cr
\hat{A}_{-1} & \hat{A}_0 & \hat{A}_1 & & \cr
& \hat{A}_{-1} & \hat{A}_0 & \hat{A}_1 & \cr
& & \ddots & \ddots & \ddots 
\end{pmatrix}. 
\end{equation}
We omit the specific description of $ \hat{A}_{-1}$,  $\hat{A}_0$ and $\hat{A}_1$. Instead,  in the case where $\hat{Z}_{0,n}=d$ and $\hat{Z}_{i,n}\ge 2$, $i\in\{1,2,...,d-1\}$, we give their description in terms of $A_{\bi_{(d)}},\,\bi_{(d)}=(i_1,i_2,...,i_d)\in\{-1,0,1\}^d$. 
For $k\in\{-1,0,1\}$ and $\bi_{(d-1)}=(i_1,i_2,...,i_{d-1})\in\mathbb{Z}^{d-1}$, define a block $\hat{A}_{k,\bi_{(d-1)}}$ as 
\[
\hat{A}_{k,\bi_{(d-1)}} = \left([\hat{A}_k]_{(d,\bx_{(d-1)},j),(d,\bx_{(d-1)}+\bi_{(d-1)},j')}; j,j'\in S_0 \right), 
\]
where we assume that $\bx_{(d-1)}=(x_1,x_2,...,x_{d-1})\ge 2\,\bone$ and use the fact that the right hand side does not depend on $\bx_{(d-1)}$ because of the space homogeneity of $\{\bY_n\}$. Since the level process $\{\bX_n\}$ of the original MMRW is skip free in all directions, the block $\hat{A}_{k,\bi_{(d-1)}}$ is given as
\begin{equation}
\hat{A}_{k,\bi_{(d-1)}} = \left\{ \begin{array}{ll}
A_{(\bi_{(d-1)}-\bone_{d-1},-1)}, & k=-1,\ \bi_{(d-1)}\in\{0,1,2\}^{d-1}, \cr
A_{(\bi_{(d-1)},0)}, & k=0,\ \bi_{(d-1)}\in\{-1,0,1\}^{d-1}, \cr
A_{(\bi_{(d-1)}+\bone_{d-1},1)}, & k=1,\ \bi_{(d-1)}\in\{-2,-1,0\}^{d-1}, \cr
O, & \mbox{otherwise}, 
\end{array} \right.
\label{eq:hatAki}
\end{equation}
where, for positive integer $l$, we denote by $\bone_l$ an $l$-dimensional vector of $1$'s (see Fig.\ref{fig:blocks_c}). 

\begin{figure}[htbp]
\begin{center}
\includegraphics[width=14cm,trim=20 250 20 20]{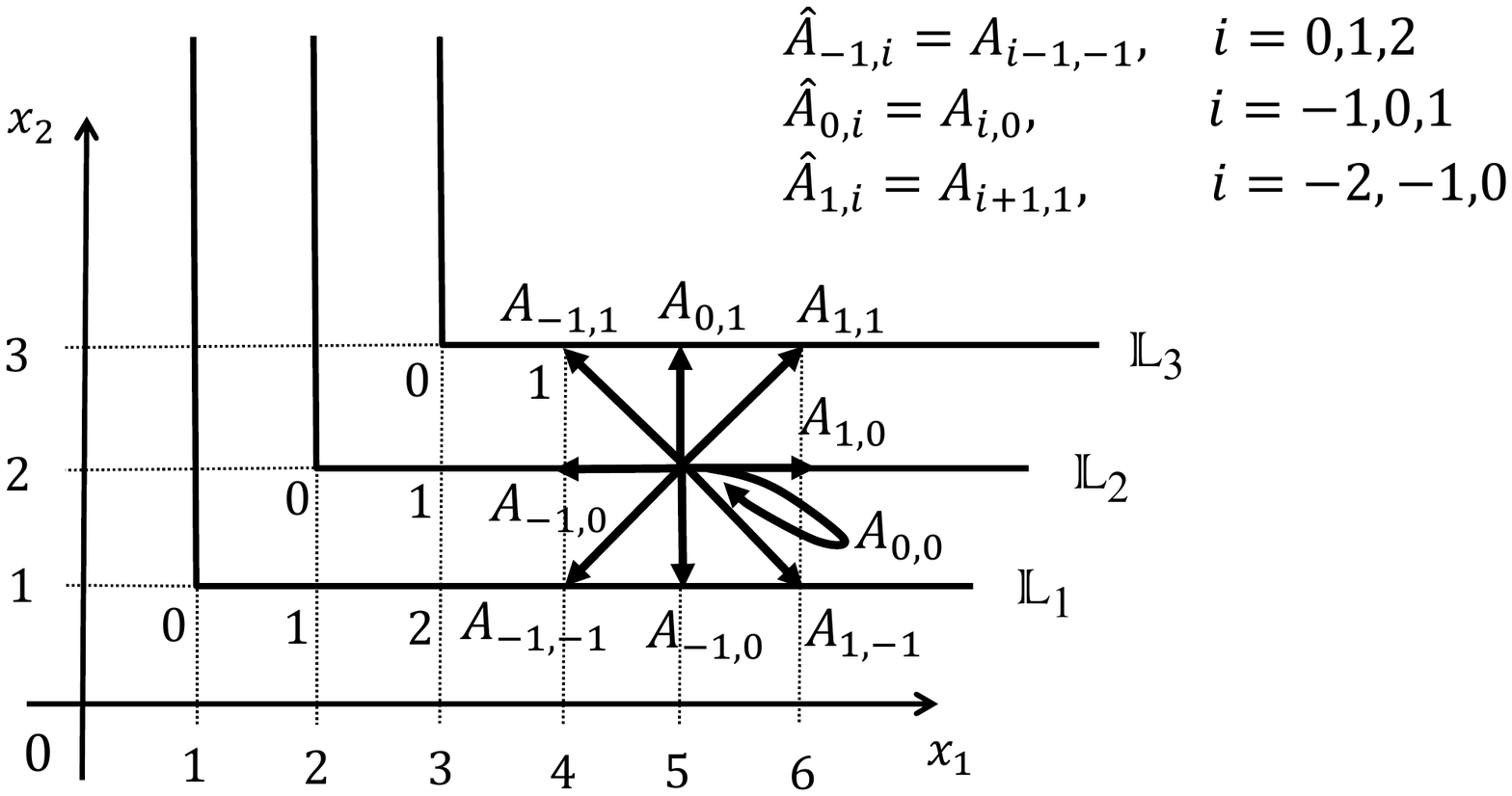} 
\caption{Transition probability blocks of $\{\hat{\bY}_n\}$ ($d=2$)}
\label{fig:blocks_c}
\end{center}
\end{figure}

%
Recall that $\tilde{P}_+$ is the fundamental matrix of the substochastic matrix $P_+$ and each row of $\tilde{P}_+$ is an occupation measure. 
For $\bx,\bx'\in\mathbb{Z}_+^d$, the matrix $N_{\bx,\bx'}$ is given as $N_{\bx,\bx'}=(\tilde{q}_{(\bx,j),(\bx',j')}; j,j'\in S_0)$.  In terms of $N_{\bx,\bx'}$, $\tilde{P}_+$ is represented as $\tilde{P}_+=(N_{\bx,\bx'}; \bx,\bx'\in\mathbb{Z}_+^d)$. 
We derive a matrix geometric representation for $\tilde{P}_+$ according to the QBD process $\{\hat{\bY}_n\}$. 
Under Assumption \ref{as:finiteness_tildeQ}, the summation of each row of $\tilde{P}_+$ is finite and we obtain the following recursive formula for $\tilde{P}_+$:
\begin{equation}
\tilde{P}_+=I+\tilde{P}_+P_+. 
\label{eq:tQQ}
\end{equation}
Define $\hat{N}_0$ as $\hat{N}_0 = \big( \hat{N}_{0,k}; k\in\mathbb{Z}_+ \big)$, where 
\[
\hat{N}_{0,k} = \big( N_{\bx,\bx'}; \bx=(x_1,...,x_d),\,\bx'=(x_1',...,x_d')\in\mathbb{Z}_+^d\ \mbox{s.t.}\  \min_{1\le i\le d} x_i=0,\ \min_{1\le i\le d} x_i'=k \big).
\]
Since $\hat{N}_0$ is a submatrix of $\tilde{P}_+$, we obtain from (\ref{eq:tQQ}) that
%
%
%
\begin{equation}
\hat{N}_0 = \begin{pmatrix} I & O & \cdots \end{pmatrix} + \hat{N}_0 \hat{P},  
\end{equation}
where $P_+$ in \eqref{eq:tQQ} is replaced with $\hat{P}$ and this $\hat{P}$ has the same block structure as $\hat{N}_0$. This equation leads us to  
\begin{equation}
\begin{array}{l}
\hat{N}_{0,0} = I + \hat{N}_{0,0} \hat{A}_0 + \hat{N}_{0,1} \hat{A}_{-1},\\[2pt]
\hat{N}_{0,k} = \hat{N}_{0,k-1} \hat{A}_1 + \hat{N}_{0,k} \hat{A}_0 + \hat{N}_{0,k+1} \hat{A}_{-1},\ k\ge 1.  
\end{array}
\label{eq:N0n}
\end{equation}
Let $\hat{R}$ be the rate matrix generated from the triplet $\{\hat{A}_{-1},\hat{A}_0,\hat{A}_1\}$, which is the minimal nonnegative solution to the matrix quadratic equation:
\begin{equation}
\hat{R} =  \hat{A}_{1} + \hat{R} \hat{A}_{0} + \hat{R}^2 \hat{A}_{-1} .
\end{equation}
Then, the solution to equations (\ref{eq:N0n}) is given as
\begin{align}
&\hat{N}_{0,k} = \hat{N}_{0,0} \hat{R}^k,\quad 
\hat{N}_{0,0} = (I-\hat{A}_0-\hat{R}\hat{A}_{-1})^{-1} = \sum_{k=0}^\infty (\hat{A}_0+\hat{R}\hat{A}_{-1})^k, 
\label{eq:N0n_solution}
\end{align}
where we use the fact that $\cp(\hat{A}_0+\hat{R}\hat{A}_{-1})<1$ since $\tilde{P}_+$ is finite. 

%
Next, we give an upper bound for $\cp(\hat{R})$, the convergence parameter of $\hat{R}$.
For $\theta\in\mathbb{R}$, define a matrix function $\hat{A}_*(\theta)$ as
\[
\hat{A}_*(\theta) = e^{-\theta} \hat{A}_{-1} + \hat{A}_0 + e^{\theta} \hat{A}_1.  
\]
Since $P_+$ is irreducible and the number of positive elements of each row and column of $\hat{A}_*(0)$ is finite, we have, by Lemma \ref{le:RandG_cp}, 
\begin{equation}
\log \cp(\hat{R}) = \sup\{\theta\in\mathbb{R}; \cp(\hat{A}_*(\theta)) > 1 \}.
\label{eq:cpR_1}
\end{equation}
We consider relation between $\cp(A_*(\btheta))$ and $\cp(\hat{A}_*(\theta))$. 
For $\bi_{(d-1)}\in\mathbb{Z}^{d-1}$, define a matrix function $\hat{A}_{*,\bi_{(d-1)}}(\theta)$ as 
\[
\hat{A}_{*,\bi_{(d-1)}}(\theta) = e^{-\theta} \hat{A}_{-1,\bi_{(d-1)}} + \hat{A}_{0,\bi_{(d-1)}} + e^{\theta} \hat{A}_{1,\bi_{(d-1)}}. 
\]
Further define a block matrix $\hat{A}_*^\dagger(\theta)$ as
\[
\hat{A}_*^\dagger(\theta)=\left(\hat{A}_{*,\bx'_{(d-1)}-\bx_{(d-1)}}(\theta); \bx_{(d-1)},\bx'_{(d-1)}\in\mathbb{Z}_+^{d-1}\right),
\]
where $\hat{A}_{*,\bx'_{(d-1)}-\bx_{(d-1)}}(\theta)= O$ if $\bx'_{(d-1)}-\bx_{(d-1)}\notin\{-2,-1,0,1,2\}^{d-1}$. 
The matrix $\hat{A}_*^\dagger(\theta)$ is a submatrix of $\hat{A}_*(\theta)$, obtained by restricting the state space $\mathbb{N}_d\times\mathbb{Z}_+^{d-1}\times S_0$ to $\{d\}\times\mathbb{Z}_+^{d-1}\times S_0$. Hence, we have
\begin{equation}
\cp(\hat{A}_*(\theta)) \le \cp(\hat{A}_*^\dagger(\theta)). 
\label{eq:hatAs_hatAdagger_cp}
\end{equation}
Define a matrix function $\hat{A}_{*,*}(\theta,\btheta_{(d-1)})$ as 
\[
\hat{A}_{*,*}(\theta,\btheta_{(d-1)}) = \sum_{\bi_{(d-1)}\in\{-2,-1,0,1,2\}^{d-1}} e^{\langle \bi_{(d-1)},\btheta_{(d-1)} \rangle} \hat{A}_{*,\bi_{(d-1)}}(\theta),  
\]
where $\btheta_{(d-1)}=(\theta_1,\theta_2,...,\theta_{d-1})$. From \eqref{eq:hatAki}, we see that $\hat{A}_*^\dagger(\theta)$ is a multiple-block-quintuple-diagonal matrix and, applying Remark \ref{re:multidiagonal_cp} to it repeatedly, we obtain 
\begin{equation}
\cp(\hat{A}_*^\dagger(\theta)) = \sup_{\btheta_{(d-1)}\in\mathbb{R}^{d-1}} \cp(\hat{A}_{*,*}(\theta,\btheta_{(d-1)})). 
\label{eq:hatAdagger_hatAss_cp}
\end{equation}
Furthermore, from \eqref{eq:hatAki}, we have
\begin{align}
&\quad \hat{A}_{*,*}(\theta+\sum_{k=1}^{d-1} \theta_k,\btheta_{(d-1)}) \cr
&= \sum_{\bi_{(d-1)}\in\{-2,-1,0,1,2\}^{d-1}} e^{\langle \bi_{(d-1)},\btheta_{(d-1)} \rangle} \left(e^{-\theta-\sum_{k=1}^{d-1} \theta_k} \hat{A}_{-1,\bi_{(d-1)}} + \hat{A}_{0,\bi_{(d-1)}} + e^{\theta+\sum_{k=1}^{d-1} \theta_k} \hat{A}_{1,\bi_{(d-1)}} \right) \cr
&= \sum_{\bi_{(d-1)}\in\{0,1,2\}^{d-1}} e^{-\theta+\langle \bi_{(d-1)}-\bone_{d-1},\btheta_{(d-1)} \rangle} A_{(\bi_{(d-1)}-\bone_{d-1},-1)} 
+ \sum_{\bi_{(d-1)}\in\{-1,0,1\}^{d-1}} e^{\langle \bi_{(d-1)},\btheta_{(d-1)} \rangle} A_{(\bi_{(d-1)},0)} \cr
&\qquad + \sum_{\bi_{(d-1)}\in\{-2,-1,0\}^{d-1}} e^{\theta+\langle \bi_{(d-1)}+\bone_{d-1},\btheta_{(d-1)} \rangle} A_{(\bi_{(d-1)}+\bone_{d-1},1)} \cr
&= A_*(\btheta_{(d-1)},\theta). 
\label{eq:hatAss_As_cp}
\end{align}
Hence, we obtain the following proposition.
\begin{proposition} \label{pr:cp_hatR_upper1}
\begin{equation}
\log\cp(\hat{R}) \le \sup\{ \langle \bone,\btheta \rangle; \cp(A_*(\btheta))>1,\ \btheta\in\mathbb{R}^d\}. 
\label{eq:cp_hatR_upper1}
\end{equation}
\end{proposition}

\begin{proof}
From \eqref{eq:hatAs_hatAdagger_cp}, \eqref{eq:hatAdagger_hatAss_cp} and \eqref{eq:hatAss_As_cp}, we obtain 
\begin{align}
\{\theta\in\mathbb{R}; \cp(\hat{A}_*(\theta)) > 1 \} 
&\subset \{\theta\in\mathbb{R}; \cp(\hat{A}_*^\dagger(\theta)) > 1 \} \cr
&= \{\theta\in\mathbb{R}; \cp(\hat{A}_{*,*}(\theta,\btheta_{(d-1)})) > 1\ \mbox{for some}\ \btheta_{(d-1)}\in\mathbb{R}^{d-1} \} \cr
&= \left\{\sum_{k=1}^d \theta_k \in\mathbb{R}; \cp\bigg(\hat{A}_{*,*}\Big(\sum_{k=1}^d \theta_k,\btheta_{(d-1)}\Big)\bigg) > 1 \right\} \cr
&= \{ \langle \bone,\btheta \rangle; \cp(A_*(\btheta))>1,\ \btheta\in\mathbb{R}^d\}.
\end{align}
This and \eqref{eq:cpR_1} lead us to inequality \eqref{eq:cp_hatR_upper1}. 
\end{proof}

%
%
\subsection{QBD representation with level direction vector $\bc$}

Letting $\bc=(c_1,c_2,...,c_d)$ be a vector of positive integers, we consider another QBD representation of $\{\bY_n\}=\{(\bX_n,J_n)\}$, whose level direction vector is given by $\bc$. 
For $k\in\{1,2,...,d\}$, denote by ${}^{\bc}\!X_{k,n}$ and ${}^{\bc}\!M_{k,n}$ the quotient and remainder of $X_{k,n}$ divided by $c_k$, respectively, i.e., 
\[
X_{k,n}=c_k {}^{\bc}\!X_{k,n}+{}^{\bc}\!M_{k,n},
\]
where ${}^{\bc}\!X_{k,n}\in\mathbb{Z}$ and ${}^{\bc}\!M_{k,n}\in\{0,1,...,c_k-1\}$. 
Define a process $\{{}^{\bc}\bY_n\}$ as 
\[
{}^{\bc}\bY_n=({}^{\bc}\!\bX_n,({}^{\bc}\!\bM_n,J_n)),
\]
where ${}^{\bc}\!\bX_n=({}^{\bc}\!X_{1,n},{}^{\bc}\!X_{2,n},...,{}^{\bc}\!X_{d,n})$ and ${}^{\bc}\!\bM_n=({}^{\bc}\!M_{1,n},{}^{\bc}\!M_{2,n},...,{}^{\bc}\!M_{d,n})$. The process $\{{}^{\bc}\bY_n\}$ is a $d$-dimensional MMRW with the background process $\{({}^{\bc}\!\bM_n,J_n)\}$ and its state space is given by $\mathbb{Z}^d\times(\prod_{k=1}^d \mathbb{Z}_{0,c_k-1}\times S_0)$, where $\mathbb{Z}_{0,c_k-1}=\{0,1,...,c_k-1\}$. 
The transition probability matrix of $\{{}^{\bc}\bY_n\}$, denoted by ${}^{\bc}\!P$, has a multiple-tridiagonal block structure like $P$. Denote by  ${}^{\bc}\!A_{\bi},\,\bi\in\{-1,0,1\}^d$, the nonzero blocks of ${}^{\bc}\!P$ and define a matrix function ${}^{\bc}\!A_*(\btheta)$ as
\[
{}^{\bc}\!A_*(\btheta) = \sum_{\bi\in\{-1,0,1\}^d} e^{\langle \bi,\btheta \rangle}\, {}^{\bc}\!A_{\bi}. 
 \]
The following relation holds between $A_*(\btheta)$ and ${}^{\bc}\!A_*(\btheta)$.
\begin{proposition} \label{pr:QBDcp}
For any vector $\bc=(c_1,c_2,...,c_d)$ of positive integers, we have 
\begin{equation}
\cp(A_*(\btheta)) = \cp({}^{\bc}\!A_*(\bc\bullet\btheta)),  
\end{equation}
where $\btheta=(\theta_1,\theta_2,...,\theta_d)$ and $\bc\bullet\btheta=(c_1\theta_1,c_2\theta_2,...,c_d\theta_d)$. 
\end{proposition}

We use the following proposition for proving Proposition \ref{pr:QBDcp}. 
\begin{proposition} \label{pr:block_cp}
Let $C_{-1}$, $C_0$ and $C_1$ be $m\times m$ nonnegative matrices, where $m$ can be countably infinite, and define a matrix function $C_*(\theta)$ as 
\begin{equation}
C_*(\theta) = e^{-\theta} C_{-1} + C_0 + e^{\theta} C_1.
\end{equation}
Assume that, for any $n\in\mathbb{Z}_+$, $C_*(0)^n$ is finite and $C_*(0)$ is irreducible. 
Let $k$ be a positive integer and define a $k\times k$ block matrix $C^{[k]}(\theta)$ as 
\begin{align}
&C^{[k]}(\theta) = 
\begin{pmatrix}
C_0 & C_1 & & & e^{-\theta} C_{-1} \cr
C_{-1} & C_0 & C_1 & & & \cr
& \ddots & \ddots & \ddots & & \cr
& & C_{-1} & C_0 & C_1 \cr
e^{\theta} C_1 & & & C_{-1} & C_0
\end{pmatrix}.
\end{align}
Then, we have $\cp(C^{[k]}(k \theta))=\cp(C_*(\theta))$. 
\end{proposition}

\begin{proof}
First, assume that, for a positive number $\beta$ and measure $\bu$, $\beta \bu C_*(\theta) \le \bu$, and define a measure $\bu^{[k]}$ as
\[
\bu^{[k]} = 
\begin{pmatrix} 
e^{(k-1)\theta} \bu & e^{(k-2)\theta} \bu & \cdots & e^{\theta} \bu & \bu
\end{pmatrix}.
\]
Then, we have $\beta \bu^{[k]} C^{[k]}(k\theta) \le \bu^{[k]}$ and, by Theorem 6.3 of \cite{Seneta06}, we obtain $\cp(C_*(\theta)) \le \cp(C^{[k]}(k\theta))$. 

Next, assume that, for a positive number $\beta$ and measure $\bu^{[k]}=\begin{pmatrix} \bu_1 & \bu_2 & \cdots & \bu_k \end{pmatrix}$, $\beta \bu^{[k]} C^{[k]}(k\theta) \le \bu^{[k]}$, and define a measure $\bu$ as 
\[
\bu = e^{-(k-1)\theta} \bu_1 + e^{-(k-2)\theta} \bu_2 + \cdots + e^{-\theta} \bu_{k-1} +\bu_k.
\]
Further, define a nonnegative matrix $V^{[k]}$ as
\[
V^{[k]} = 
\begin{pmatrix}
e^{-(k-1)\theta} I & e^{-(k-2)\theta} I  & \cdots & e^{-\theta} I & I
\end{pmatrix}.
\]
Then, we have $\beta \bu^{[k]} C^{[k]}(k\theta) V^{[k]} = \beta \bu C_*(\theta)$ and $\bu^{[k]} V^{[k]} = \bu$. Hence, we have $\beta \bu C_*(\theta) \le \bu$ and this implies $\cp(C^{[k]}(k\theta)) \le \cp(C_*(\theta))$.
\end{proof}

%
\begin{proof}[Proof of Proposition  \ref{pr:QBDcp}]
Let  $\btheta=(\theta_1,\theta_2,...,\theta_d)$ be a $d$-dimensional vector in $\mathbb{R}^d$ and, for $k\in\{1,2,...,d\}$, define $\btheta_{(k)}$ and $\btheta_{[k]}$ as $\btheta_{(k)}=(\theta_1,\theta_2,...,\theta_k)$ and $\btheta_{[k]}=(\theta_k,\theta_{k+1},...,\theta_d)$, respectively. 
We consider the multiple-block structure of ${}^{\bc}\!A_*(\btheta)$ according to $\mathbb{Z}_{0,c_1-1}\times\mathbb{Z}_{0,c_2-1}\times \cdots \times \mathbb{Z}_{0,c_d-1}\times S_0$,  the state space of the background process of $\{{}^{\bc}\bY_n\}$. 
For $k\in\{-1,0,1\}$, define ${}^{\bc}\!A_k^{[1]}(\btheta_{[2]})$ as
\[
{}^{\bc}\!A_k^{[1]}(\btheta_{[2]}) = \sum_{\bi_{[2]}\in\{-1,0,1\}^{d-1}} e^{\langle \bi_{[2]},\btheta_{[2]} \rangle}\, {}^{\bc}\!A_{(k,\bi_{[2]})}
\]
where $\bi_{[2]}=(i_2,i_3,...,i_d)$. Due to the skip-free property of the original process, they are given in $c_1\times c_1$ block form as
\begin{align*}
& {}^{\bc}\!A_0^{[1]}(\btheta_{[2]}) = 
\begin{pmatrix}
B^{[1]}_{0}(\btheta_{[2]}) & B^{[1]}_{1}(\btheta_{[2]}) & & &  \cr
B^{[1]}_{-1}(\btheta_{[2]}) & B^{[1]}_{0}(\btheta_{[2]}) & B^{[1]}_{1}(\btheta_{[2]}) & & \cr
& \ddots & \ddots & \ddots & & \cr
& & B^{[1]}_{-1}(\btheta_{[2]}) & B^{[1]}_{0}(\btheta_{[2]}) & B^{[1]}_{1}(\btheta_{[2]}) \cr
& & & B^{[1]}_{-1}(\btheta_{[2]}) & B^{[1]}_{0}(\btheta_{[2]})
\end{pmatrix}, \cr
& {}^{\bc}\!A_{-1}^{[1]}(\btheta_{[2]}) = 
\begin{pmatrix}
& & B^{[1]}_{-1}(\btheta_{[2]}) \cr
& & \cr
& O& 
\end{pmatrix}, \quad
{}^{\bc}\!A_1^{[1]}(\btheta_{[2]}) = 
\begin{pmatrix}
& O & \cr
& & \cr
B^{[1]}_{1}(\btheta_{[2]}) & &
\end{pmatrix}, 
\end{align*}
where each $B^{[1]}_{i}(\btheta_{[2]})$ is a matrix function of $\btheta_{[2]}$ and we use the fact that ${}^{\bc}\!M_{1,n}$ is the remainder of $X_{1,n}$ divided by $c_1$. Hence, ${}^{\bc}\!A_*(\btheta)$ is given in $c_1\times c_1$ block form as
\begin{align}
{}^{\bc}\!A_*(\btheta) &= 
e^{-\theta_1}\, {}^{\bc}\!A_{-1}^{[1]}(\btheta_{[2]}) + {}^{\bc}\!A_0^{[1]}(\btheta_{[2]}) + e^{\theta_1}\, {}^{\bc}\!A_1^{[1]}(\btheta_{[2]}) \cr
&= \begin{pmatrix}
B^{[1]}_{0}(\btheta_{[2]}) & B^{[1]}_{1}(\btheta_{[2]}) & & & e^{-\theta_1} B^{[1]}_{-1}(\btheta_{[2]}) \cr
B^{[1]}_{-1}(\btheta_{[2]}) & B^{[1]}_{0}(\btheta_{[2]}) & B^{[1]}_{1}(\btheta_{[2]}) & & \cr
& \ddots & \ddots & \ddots & & \cr
& & B^{[1]}_{-1}(\btheta_{[2]}) & B^{[1]}_{0}(\btheta_{[2]}) & B^{[1]}_{1}(\btheta_{[2]}) \cr
e^{\theta_1} B^{[1]}_{1}(\btheta_{[2]}) & & & B^{[1]}_{-1}(\btheta_{[2]}) & B^{[1]}_{0}(\btheta_{[2]})
\end{pmatrix}.  
\label{eq:cAs}
\end{align}
Define a matrix function $B^{[1]}_*(\theta_1,\btheta_{[2]})$ as
\begin{equation}
B^{[1]}_*(\theta_1,\btheta_{[2]}) = e^{-\theta_1} B^{[1]}_{-1}(\btheta_{[2]})+B^{[1]}_{0}(\btheta_{[2]})+e^{\theta_1} B^{[1]}_{0}(\btheta_{[2]}). 
\label{eq:B1s}
\end{equation}
Then, by Proposition \ref{pr:block_cp}, we have
\begin{equation}
\cp({}^{\bc}\!A_*(c_1\theta_1,\btheta_{[2]}))=\cp(B^{[1]}_*(\theta_1,\btheta_{[2]})). 
\end{equation}
Analogously, for $i_1\in\{-1,0,1\}$, $B^{[1]}_{i_1}(\btheta_{[2]})$ is represented in $c_2\times c_2$ block form as
\begin{equation}
B^{[1]}_{i_1}(\btheta_{[2]}) = 
\begin{pmatrix}
B^{[2]}_{i_1,0}(\btheta_{[3]}) & B^{[2]}_{i_1,1}
(\btheta_{[3]}) & & & e^{-\theta_2} B^{[2]}_{i_1,-1}(\btheta_{[3]}) \cr
B^{[2]}_{i_1,-1}(\btheta_{[3]}) & B^{[2]}_{i_1,0}(\btheta_{[3]}) & B^{[2]}_{i_1,1}(\btheta_{[3]}) & & \cr
& \ddots & \ddots & \ddots & & \cr
& & B^{[2]}_{i_1,-1}(\btheta_{[3]}) & B^{[2]}_{i_1,0}(\btheta_{[3]}) & B^{[2]}_{i_1,1}(\btheta_{[3]}) \cr
e^{\theta_2} B^{[2]}_{i_1,1}(\btheta_{[3]}) & & & B^{[2]}_{i_1,-1}(\btheta_{[3]}) & B^{[2]}_{i_1,0}(\btheta_{[3]})
\end{pmatrix}, 
\label{eq:B1i1}
\end{equation}
where each $B^{[2]}_{i_1,i_2}(\btheta_{[3]})$ is a matrix function of $\btheta_{[3]}$. Define a matrix function $B^{[2]}_*(\theta_1,\theta_2,\btheta_{[3]})$ as
\[
B^{[2]}_*(\theta_1,\theta_2,\btheta_{[3]}) = \sum_{i_1,i_2\in\{-1,0,1\}} e^{i_1\theta_1+i_2\theta_2} B^{[2]}_{i_1,i_2}(\btheta_{[3]}). 
\]
Then, by Proposition \ref{pr:block_cp}, we obtain from \eqref{eq:B1s} and \eqref{eq:B1i1} that 
\begin{equation}
\cp(B^{[1]}_*(\theta_1,c_2 \theta_2,\btheta_{[3]})) = \cp(B^{[2]}_*(\theta_1,\theta_2,\btheta_{[3]})). 
\end{equation}
Repeating this procedure more $(d-3)$ times, we obtain
\begin{equation}
\cp(B^{[d-1]}_*(\btheta_{(d-1)},c_d \theta_d)) = \cp(B^{[d]}_*(\btheta_{(d-1)},\theta_d)), 
\end{equation}
where 
\begin{align*}
&B^{[d]}_*(\btheta_{(d-1)},\theta_d) = \sum_{\bi_{(d-1)}\in\{-1,0,1\}^{d-1}}\, \sum_{i_d\in\{-1,0,1\}} e^{\langle \bi_{(d-1)},\btheta_{(d-1)} \rangle+i_d\theta_d} B^{[d]}_{\bi_{(d-1)},i_d}, \cr
&B^{[d]}_{\bi_{(d-1)},i_d} = A_{(\bi_{(d-1)},i_d)},
\end{align*}
and $\bi_{(d-1)}=(i_1,i_2,...,i_{d-1})$. As a result, we have 
\begin{align}
\cp({}^{\bc}\!A_*(\bc\bullet\btheta))
&=\cp(B^{[1]}_*(\theta_1,\bc_{[2]}\bullet\btheta_{[2]})) \cr
&=\cp(B^{[2]}_*(\btheta_{(2)},\bc_{[3]}\bullet\btheta_{[3]})) \cr
&\quad \cdots \cr
&= \cp(B^{[d-1]}_*(\btheta_{(d-1)},c_d\theta_d)) = \cp(A_*(\btheta)), 
\end{align}
where $\bc_{[k]}=(c_k,c_{k+1},...,c_d)$, and this completes the proof. 
\end{proof}

%
Next, we apply the results of the previous subsection to the $d$-dimensional MMRW $\{{}^{\bc}\bY_n\}$. Let $\{{}^{\bc}\hat{\bY}_n\}$ be a one-dimensional absorbing QBD process with level direction vector $\bone$, generated from $\{{}^{\bc}\bY_n\}$. The process $\{{}^{\bc}\hat{\bY}_n\}$ is given as 
\[
{}^{\bc}\hat{\bY}_n = ({}^{\bc}\!\hat{X}_n, ({}^{\bc}\!\hat{\bZ}_n,{}^{\bc}\!\hat{\bM}_n,\hat{J}_n) ),
\]
where
\begin{align*}
& {}^{\bc}\!\hat{X}_n = \min_{1\le i\le d} {}^{\bc}\!X_{i,\tau\wedge n},\\
& {}^{\bc}\!\hat{\bZ}_n = ({}^{\bc}\!\hat{Z}_{0,n},{}^{\bc}\!\hat{Z}_{1,n},...,{}^{\bc}\!\hat{Z}_{d-1,n}), \\
& {}^{\bc}\!\hat{Z}_{0,n}=\min\!\big\{i\in\{1,2,...,d\}; {}^{\bc}\!X_{i,\tau\wedge n}={}^{\bc}\!\hat{X}_n \big\}, \\
& {}^{\bc}\!\hat{Z}_{i,n} = \left\{ \begin{array}{ll}
{}^{\bc}\!X_{i,\tau\wedge n}-{}^{\bc}\!\hat{X}_n, & i<{}^{\bc}\!\hat{Z}_{0,n}, \cr
{}^{\bc}\!X_{i+1,\tau\wedge n}-{}^{\bc}\!\hat{X}_n, & i\ge {}^{\bc}\!\hat{Z}_{0,n}, 
\end{array} \right. 
i=1,2,...,d-1,  \\
& {}^{\bc}\!\hat{\bM}_n=({}^{\bc}\!\hat{M}_{1,n},{}^{\bc}\!\hat{M}_{2,n},...,{}^{\bc}\!\hat{M}_{d,n})
=({}^{\bc}\!M_{1,\tau\wedge n},{}^{\bc}\!M_{2,\tau\wedge n},...,{}^{\bc}\!M_{d,\tau\wedge n}), \\
&\hat{J}_n=J_{\tau\wedge n},  
\end{align*}
and $\tau$ is the stopping time at which the original MMRW $\{\bY_n\}$ enters $\mathbb{S}\setminus\mathbb{S}_+$ for the first time. 
We restrict the state space of $\{{}^{\bc}\hat{\bY}_n\}$ to $\mathbb{Z}_+\times(\mathbb{N}_d\times\mathbb{Z}_+^{d-1}\times\prod_{k=1}^d \mathbb{Z}_{0,c_k-1}\times S_0)$. 
For $k\in\mathbb{Z}_+$, the $k$-th level set of $\{{}^{\bc}\hat{\bY}_n\}$ is given by 
\begin{equation}
{}^{\bc}\mathbb{L}_k = \left\{ (\bx,j)\in\mathbb{Z}_+^d\times S_0; \min_{1\le i\le d} \lfloor x_i/c_i \rfloor = k \right\},
\end{equation}
where $\lfloor x \rfloor$ is the maximum integer less than or equal to $x$. The level sets satisfy, for $k\ge 0$,  
\begin{equation}
{}^{\bc}\mathbb{L}_{k+1} = \{(\bx+\bc,j); (\bx,j)\in{}^{\bc}\mathbb{L}_k \}.
\end{equation}
This means that $\{{}^{\bc}\hat{\bY}_n\}$ is a QBD process with level direction vector $\bc$. Let ${}^{\bc}\!\hat{R}$ be the rate matrix of the QBD process $\{{}^{\bc}\hat{\bY}_n\}$. An upper bound for the convergence parameter of ${}^{\bc}\!\hat{R}$ is given as follows.
\begin{lemma} \label{le:cp_hatcR_upper1}
\begin{equation}
\log\cp({}^{\bc}\!\hat{R}) \le \sup\{ \langle \bc,\btheta \rangle; \cp(A_*(\btheta))>1,\ \btheta\in\mathbb{R}^d\}. 
\label{eq:cp_hatcR_upper1}
\end{equation}
\end{lemma}

\begin{proof}
By Propositions \ref{pr:cp_hatR_upper1} and \ref{pr:QBDcp}, we have
\begin{align*}
\log\cp({}^{\bc}\!\hat{R}) 
&\le \sup\{ \langle \bone,\btheta \rangle; \cp({}^{\bc}\!A_*(\btheta))>1,\ \btheta\in\mathbb{R}^d\} \cr
&= \sup\{ \langle \bone,\bc\bullet\btheta \rangle; \cp({}^{\bc}\!A_*(\bc\bullet\btheta))>1,\ \btheta\in\mathbb{R}^d\} \cr
&= \sup\{ \langle \bc,\btheta \rangle; \cp(A_*(\btheta))>1,\ \btheta\in\mathbb{R}^d\}.
\end{align*}
\end{proof}

%
%
%
\section{Asymptotic property of the occupation measures} \label{sec:asymptotic}

In this section, we derive the asymptotic decay rates of the occupation measures in the $d$-dimensional MMRW $\{\bY_n\}=\{(\bX_n,J_n)\}$. We also obtain the convergence domains of the matrix moment generating functions for the occupation measures.

\subsection{Asymptotic decay rate in an arbitrary direction}

Recall that, for $\bx\in\mathbb{Z}_+^d$, the convergence domain of the matrix moment generating function $\Phi_{\bx}(\btheta)$ is given as $\calD_{\bx} = \mbox{the interior of }\{\btheta\in\mathbb{R}^d : \Phi_{\bx}(\btheta)<\infty\}$. This domain does not depend on $\bx$, as follows. 
\begin{proposition} \label{pr:Dx=Dxp}
For every $\bx,\bx'\in\mathbb{Z}_+^d$, $\calD_{\bx}=\calD_{\bx'}$. 
\end{proposition}
\begin{proof}
For every $\bx,\bx'\in\mathbb{Z}_+^d$ and $j\in S_0$, since $P_+$ is irreducible, there exists $n_0\ge 0$ such that $\mathbb{P}(\bY_{n_0}=(\bx',j)\,|\,\bY_0=(\bx,j))>0$. Using this $n_0$, we obtain, for every $j'\in S_0$, 
\begin{align}
[\Phi_{\bx}(\btheta)]_{j,j'} 
&= \mathbb{E}\bigg( \sum_{n=0}^\infty e^{\langle \bX_n,\btheta \rangle}\, 1(J_n=j')\, 1(\tau>n) \,\Big|\, \bY_0=(\bx,j) \bigg) \cr
&\ge \mathbb{E}\bigg( \sum_{n=n_0}^\infty e^{\langle \bX_n,\btheta \rangle}\, 1(J_n=j')\, 1(\tau>n) \,\Big|\, \bY_{n_0}=(\bx',j) \bigg) \mathbb{P}(\bY_{n_0}=(\bx',j)\,|\,\bY_0=(\bx,j)) \cr
&= [\Phi_{\bx'}(\btheta)]_{j,j'}\, \mathbb{P}(\bY_{n_0}=(\bx',j)\,|\,\bY_0=(\bx,j)), 
\end{align}
where $\tau$ is the stopping time given as $\tau=\inf\{n\ge 0; \bY_n\in\mathbb{S}\setminus\mathbb{S}_+\}$. This implies $\calD_{\bx}\subset \calD_{\bx'}$. 
Exchanging $\bx$ with $\bx'$, we obtain $\calD_{\bx'}\subset \calD_{\bx}$, and this completes the proof. 
\end{proof}

A relation between the point sets $\Gamma$ and $\calD_x$ is given as follows.
\begin{proposition} \label{pr:domain_Gamma}
For every $\bx\in\mathbb{Z}_+^d$, $\Gamma\subset\calD_{\bx}$ and hence, $\calD\subset\calD_{\bx}$. 
\end{proposition}

\begin{proof}
If $\btheta\in\Gamma$, then $\cp(A_*(\btheta))>1$ and we have $ \sum_{k=0}^\infty A_*(\btheta)^k<\infty$. This leads us to that, for every $j,j'\in S_0$,  
\begin{align}
\infty>\bigg[ \sum_{k=0}^\infty A_*(\btheta)^k\bigg]_{j,j'} 
&= \mathbb{E}\bigg( \sum_{n=0}^\infty e^{\langle \bX_n,\btheta \rangle}\, 1(J_n=j') \,\Big|\, \bY_0=(\bzero,j) \bigg) \cr
&\ge \mathbb{E}\bigg( \sum_{n=0}^\infty e^{\langle \bX_n,\btheta \rangle}\, 1(J_n=j')\, 1(\tau>n) \,\Big|\, \bY_0=(\bzero,j) \bigg) \cr
&= [\Phi_{\bzero}(\btheta)]_{j,j'}, 
\end{align}
and we have $\Gamma\subset\calD_{\bzero}$. Hence, by Proposition \ref{pr:Dx=Dxp}, we obtain the desired result. 
\end{proof}

%
Using Lemmas \ref{le:limsup_tildeqnn} and \ref{le:cp_hatcR_upper1}, we obtain the asymptotic decay rates of the occupation measures, as follows. 
\begin{theorem} \label{th:asymptotic_any_direction}
For any positive vector $\bc=(c_1,c_2,...,c_d)\in\mathbb{Z}_+^d$, for every $\bx=(x_1,x_2,...,x_d)\in\mathbb{Z}_+^d$ such that $\min_{1\le i\le d} x_i=0$, for every $\bl=(l_1,l_2,...,l_d)\in\mathbb{Z}_+^d$ such that $\min_{1\le i\le d} l_i=0$ and for every $j,j'\in S_0$, 
\begin{align}
\lim_{k\to\infty} \frac{1}{k} \log \tilde{q}_{(\bx,j),(k\bc+\bl,j')} 
&= - \sup_{\btheta\in\Gamma}  \langle \bc, \btheta \rangle.
\label{eq:liminf_tildeq}
\end{align}
\end{theorem}

\begin{proof}
By Lemma \ref{le:limsup_tildeqnn} and Proposition \ref{pr:domain_Gamma}, we have, for any positive vector $\bc\in\mathbb{Z}_+^d$ and for every $(\bx,j)\in\mathbb{S}_+$, $j'\in S_0$ and $\bl\in\mathbb{Z}_+$,  
\begin{equation}
 \limsup_{k\to\infty} \frac{1}{k} \log \tilde{q}_{(\bx,j),(k \bc+\bl,j')} 
 \le - \sup_{\btheta\in\calD_{\bx}} \langle \bc,\btheta \rangle
 \le - \sup_{\btheta\in\Gamma} \langle \bc,\btheta \rangle. 
\end{equation}
Hence, in order to prove the theorem, it suffices to give the lower bound. 

Consider the one-dimensional QBD process $\{{}^{\bc}\hat{\bY}_n\}$ defined in the previous section. Applying Corollary \ref{co:cpRlimit} to the rate matrix $\{{}^{\bc}\!\hat{R}\}$ of $\{{}^{\bc}\hat{\bY}_n\}$, we obtain, for some $\bz''=(i'',\bx'',\bm'',j'')\in\mathbb{N}_d\times\mathbb{Z}_+^{d-1}\times\prod_{k=1}^d \mathbb{Z}_{0,c_k-1}\times S_0$ and every $\bz'=(i',\bx',\bm',j')\in\mathbb{N}_d\times\mathbb{Z}_+^{d-1}\times\prod_{k=1}^d \mathbb{Z}_{0,c_k-1}\times S_0$, 
\begin{equation}
\lim_{k\to\infty} \left( [({}^{\bc}\!\hat{R})^k]_{\bz'',\bz'} \right)^{\frac{1}{k}} = \cp({}^{\bc}\!\hat{R})^{-1},
\label{eq:charR_limit}
\end{equation}
where $\bx'=(x_1',...,x_{d-1}'), \bx''=(x_1',...,x_{d-1}'')\in\mathbb{Z}_+^{d-1}$ and $\bm'=(m_1',...,m_d'), \bm''=(m_1'',...,m_d'')\in\prod_{k=1}^d \mathbb{Z}_{0,c_k-1}$. 
For $k\ge 0$, ${}^{\bc}\hat{\bY}_n=(k,i',\bx',\bm',j')$ corresponds to $\bY_n=(k\bc+\bc\bullet\hat{\bx}'+\bm',j')$, where $\hat{\bx}'=(x_1',...,x_{i'-1}',0,x_{i'}',...,x_{d-1}')$. Analogously, ${}^{\bc}\hat{\bY}_n=(0,i'',\bx'',\bm'',j'')$ corresponds to $\bY_n=(\bc\bullet\hat{\bx}'' +\bm'',j'')$, where $\hat{\bx}''=(x_1'',...,x_{i''-1}'',0,x_{i''}'',...,x_{d-1}'')$.
Hence, from \eqref{eq:N0n_solution}, setting $\bl=\bc\bullet\hat{\bx}' +\bm'$, we obtain, for every $\bx=(x_1,x_2,...,x_d)\in\mathbb{Z}_+^d$ such that $\min_{1\le i\le d} x_i=0$ and for every $j\in S_0$,
\begin{equation}
\tilde{q}_{(\bx,j),(k\bc+\bl,j')} \ge \tilde{q}_{(\bx,j),(\bc\bullet\hat{\bx}''+\bm'',j'')} [{}^{\bc}\!\hat{R}^k]_{\bz'',\bz'}. 
\label{eq:tildeq_under1}
\end{equation}
From  \eqref{eq:charR_limit}, \eqref{eq:tildeq_under1} and \eqref{eq:cp_hatcR_upper1}, setting $\bm'=\bzero$, we obtain 
\begin{equation}
 \liminf_{k\to\infty} \frac{1}{k} \log \tilde{q}_{(\bx,j),(k \bc+\bl,j')} 
 \ge - \log\cp({}^{\bc}\!\hat{R}) 
 \ge - \sup_{\btheta\in\Gamma} \langle \bc,\btheta \rangle,  
\end{equation}
and this completes the proof. 
\end{proof}

%
\begin{corollary} \label{co:asymptotic_any_direction}
The same result as Theorem \ref{th:asymptotic_any_direction} holds for every direction vector $\bc\in\mathbb{Z}_+^d$ such that $\bc\ne\bzero$. 
\end{corollary}

\begin{proof}
Let $\{\bY_n\}=\{(\bX_n,J_n)\}$ be a $d$-dimensional MMRW on the state space $\mathbb{Z}^d\times S_0$ and define an absorbing Markov chain $\{\hat{\bY}_n\}=\{(\hat{\bX}_n,\hat{J}_n)\}$ as $\hat{\bY}_n=\bY_{\tau\wedge n}$ for $n\ge 0$, where $\tau$ is the stopping time given as $\tau=\inf\{n\ge 0; \bY_n\in\mathbb{S}\setminus\mathbb{S}_+\}$. We assume that the state space of $\{\hat{\bY}_n\}$ is given by $\mathbb{S}_+$. 
If $d=1$, the assertion of the corollary is trivial. Hence, we assume $d\ge 2$ and set $m$ in $\{1,2,...,d-1\}$. 
Without loss of generality, we assume the direction vector $\bc=(c_1,c_2,...,c_d)$ satisfies $c_i>0$ for $i\in\{1,2,...,m\}$ and $c_i=0$ for $i\in\{m+1,m+2,...,d\}$. 
Consider an $m$-dimensional MMRW $\{\hat{\bY}_n^{(m)}\}=\{(\hat{X}_1,...,\hat{X}_m,(\hat{X}_{m+1},...,\hat{X}_d,\hat{J}_n))\}$, where $(\hat{X}_1,...,\hat{X}_m)$ is the level and $(\hat{X}_{m+1},...,\hat{X}_d,\hat{J}_n)$ the background state, and denote by $A_{\bi}^{(m)},\,\bi\in\{-1,0,1\}^m$, its transition probability blocks.  
For $\btheta_{(m)}=(\theta_1,...,\theta_m)\in\mathbb{R}^m$, define a matrix function $A_*^{(m)}(\btheta_{(m)})$ as 
\[
A_*^{(m)}(\btheta_{(m)}) = \sum_{\bi\in\{-1,0,1\}^m} e^{\langle \bi,\btheta_{(m)} \rangle} A_{\bi}^{(m)}. 
\]
Since $\{\bY_n\}$ is a MMRW, this $A_*^{(m)}(\btheta_{(m)})$ has a multiple tri-diagonal structure and, applying Lemma \ref{le:Q_cp} repeatedly, we obtain 
\begin{equation}
\cp(A_*^{(m)}(\btheta_{(m)})) = \sup_{\btheta_{[m+1]}\in\mathbb{R}^{d-m}} \cp(A_*(\btheta_{(m)},\btheta_{[m+1]})),
\end{equation}
where $\btheta_{[m+1]}=(\theta_{m+1},...,\theta_d)$ and $A_*(\btheta)=A_*(\btheta_{(m)},\btheta_{[m+1]})$ is given by \eqref{eq:As}. 
Hence, applying Theorem \ref{th:asymptotic_any_direction} to $\{\hat{\bY}_n^{(m)}\}$, we obtain, for every $\bx_{(m)}=(x_1,...,x_m)\in\mathbb{Z}_+^m$ such that $\min_{1\le i\le m} x_i=0$, for every $\bx_{[m+1]}=(x_{m+1},...,x_d)\in\mathbb{Z}_+^{d-m}$,  for every $\bl_{(m)}=(l_1,...,l_m)\in\mathbb{Z}_+^m$ such that $\min_{1\le i\le m} l_i=0$ and  for every $\bl_{[m+1]}=(l_{[m+1]},...,l_d)\in\mathbb{Z}_+^{d-m}$, 
\begin{align}
&\quad \lim_{k\to\infty} \frac{1}{k} \log \tilde{q}_{(\bx_{(m)},\bx_{[m+1]},j),(k\bc_{(m)}+\bl_{(m)},\bl_{[m+1]},j')} \cr
&= -\sup\{ \langle \bc_{(m)},\btheta_{(m)}\rangle; \cp(A_*^{(m)}(\btheta_{(m)})) > 1,\, \btheta_{(m)}\in\mathbb{R}^m \} \cr
&= -\sup\Big\{ \langle \bc_{(m)},\btheta_{(m)} \rangle; \sup_{\btheta_{[m+1]}\in\mathbb{R}^{d-m}} \cp(A_*(\btheta_{(m)},\btheta_{[m+1]})) > 1,\, \btheta_{(m)}\in\mathbb{R}^m \Big\} \cr
&= - \sup_{\btheta\in\Gamma}  \langle \bc, \btheta \rangle, 
\end{align}
where $\bc_{(m)}=(c_1,...,c_m)$ and we use the assumption that $(c_{m+1},...,c_d)=\bzero$. 
\end{proof}

%
%
\subsection{Convergence domains of the matrix moment generating functions}

%
From Proposition \ref{pr:domain_Gamma} and Theorem \ref{th:asymptotic_any_direction}, we obtain the following result for the convergence domains.
\begin{theorem} \label{th:domain_Phix}
For every $\bx\in\mathbb{Z}_+^d$, $\calD_{\bx}=\calD$.
\end{theorem}
\begin{figure}[htbp]
\begin{center}
\includegraphics[width=15cm,trim=20 290 20 20]{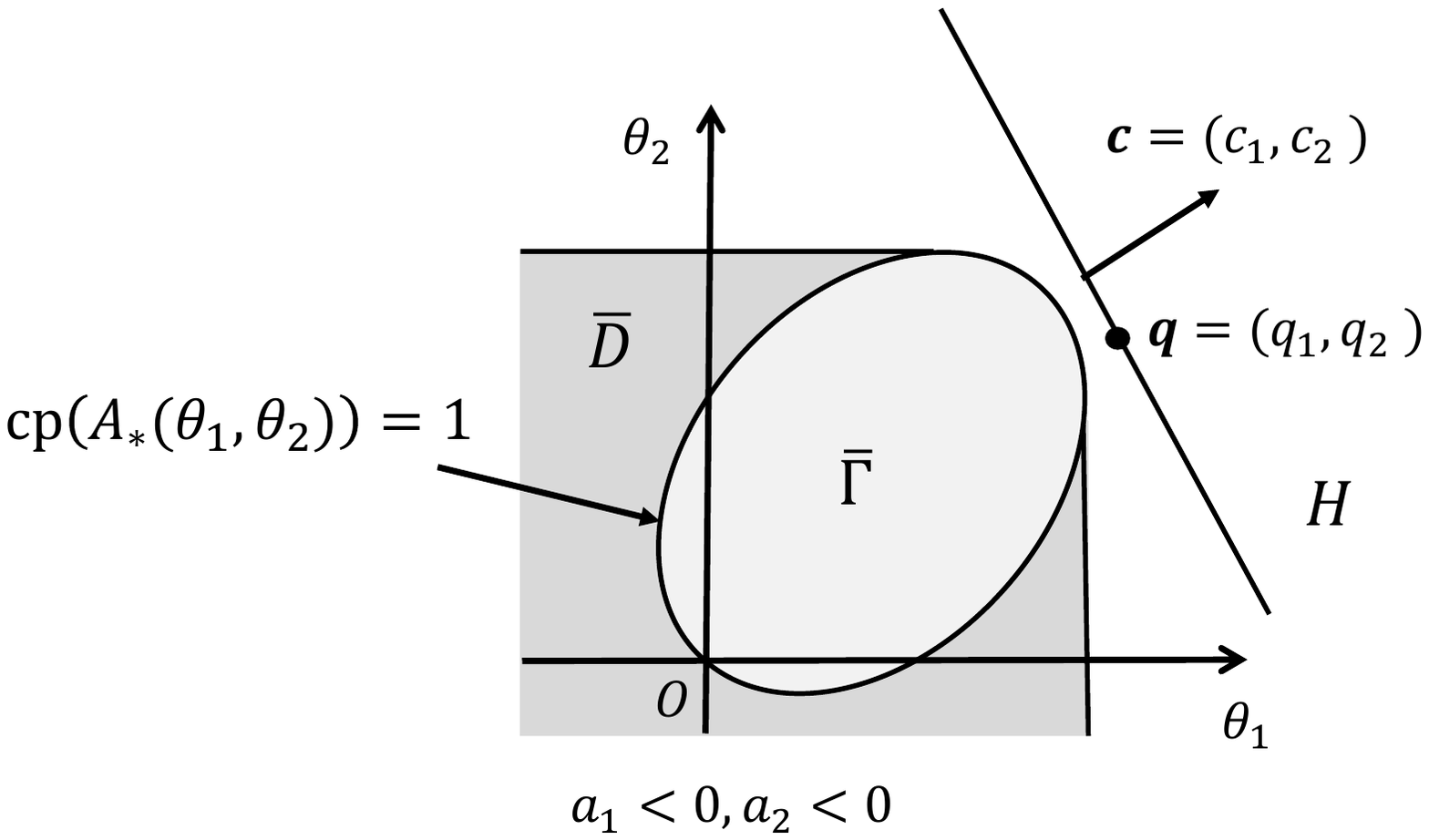} 
\caption{Convergence domain of $\Phi_{\bzero}(\btheta)$ ($d=2$)}
\label{fig:conv_domain_2}
\end{center}
\end{figure}
\begin{proof}
We prove $\calD_{\bzero}=\calD$. By Proposition \ref{pr:Dx=Dxp}, this implies $\calD_{\bx}=\calD$ for every $\bx\in\mathbb{Z}_+^d$. 
Suppose $\calD_{\bzero} \setminus\calD\ne\emptyset$. Since $\calD_{\bzero}$ is an open set and, by Proposition \ref{pr:domain_Gamma}, we have $\calD\subset\calD_{\bzero}$, there exists a point $\bq \in \calD_{\bzero}\setminus\bar{\calD}$, where $\bar{\calD}$ is the closure of $\calD$. This $\bq$ satisfies $\Phi_{\bzero}(\bq)<\infty$. 
Since $\bar{\calD}$ is a convex set, there exists a hyperplane $\mathscr{H}$ satisfying $\bq\in\mathscr{H}$ and $\bar{\calD}\cap\mathscr{H}=\emptyset$. Denote by $\bc\ge \bzero$ the normal vector of $\mathscr{H}$, where we assume $\| \bc \|=1$. By the definition, $\bc$ satisfies 
\begin{equation}
\langle \bc,\bq \rangle > \sup_{\btheta\in\calD} \langle \bc,\btheta \rangle.
\label{eq:cq_inequlity}
\end{equation}
Let $\bc'$ be a vector of positive integers satisfying 
\begin{equation}
\langle \|\bc'\|\bc,\bq \rangle > \langle \bc',\bq \rangle > \sup_{\btheta\in\calD} \langle \bc',\btheta \rangle. 
\label{eq:cpq_inequlity}
\end{equation}
It is possible because of \eqref{eq:cq_inequlity} and of the fact that $\bar{\calD}$ is bounded in any positive direction. For this $\bc'$ and for $j,j'\in S_0$, define a moment generating function $\varphi_{\bc'}(\btheta)$ as 
\begin{equation}
\varphi_{\bc'}(\btheta) = \sum_{k=0}^\infty e^{\langle \bc',\btheta \rangle k}\, \tilde{q}_{(\bzero,j),(k \bc',j')}
\label{eq:varphic}
\end{equation}
and a point ${}^{\bc'}\!\btheta$ as ${}^{\bc'}\!\btheta = \arg\max_{\btheta\in\bar{\calD}}\, \langle \bc',\btheta \rangle$. 
By Theorem \ref{th:asymptotic_any_direction} and the Cauchy-Hadamard theorem, we see that the radius of convergence of the power series in the right hand side of \eqref{eq:varphic} is $e^{\langle \bc',{}^{\bc'}\!\btheta \rangle}$ and this implies that $\varphi_{\bc}(\btheta)$ diverges if $\langle \bc',\btheta \rangle > \langle \bc',{}^{\bc'}\!\btheta \rangle$. Hence, by \eqref{eq:cpq_inequlity}, we have $\varphi_{\bc'}(\bq)=\infty$. 
On the other hand, we obtain from the definition of $\varphi_{\bc}(\btheta)$ that 
\begin{equation}
\varphi_{\bc'}(\bq) \le \sum_{\bk\in\mathbb{Z}_+^d} e^{\langle \bk,\bq \rangle}\, \tilde{q}_{(\bzero,j),(\bk,j')} 
= [\Phi_{\bzero}(\bq)]_{j,j'} < \infty.
\label{eq:varphic_Phi}
\end{equation}
This is a contradiction and, as a result, we obtain $\calD_{\bzero}\setminus\calD=\emptyset$. 
\end{proof}

%
%
\subsection{Asymptotic decay rates of marginal measures}

Let $\bX$ be a vector of random variables subject to the stationary distribution of  a multi-dimensional reflected random walk. The asymptotic decay rate of the marginal tail distribution in a form $\mathbb{P}(\langle \bc,\bX \rangle>x)$ has been discussed in \cite{Miyazawa11} (also see \cite{Kobayashi14}), where $\bc$ is a direction vector. 
In this subsection, we consider this type of asymptotic decay rate for the occupation measures. 

Let $\bc=(c_1,c_2,...,c_d)$ be a vector of mutually prime positive integers. We assume $c_1=\min_{1\le i \le d} c_i$; in other cases, analogous results can be obtained. For $k\ge 0$, define an index set $\scrI_k$ as
\[
\scrI_k=\{ \bl_{[2]}=(l_2,l_3,...,l_d)\in\mathbb{Z}_+^{d-1}; \langle \bc,(l_1,\bl_{[2]}) \rangle = c_1 k\ \mbox{for some}\ l_1\in\mathbb{Z}_+ \}. 
\]
For $\bx\in\mathbb{Z}_+^d$, the matrix moment generating function $\Phi_{\bx}(\theta \bc)$ is represented as 
\begin{equation}
\Phi_{\bx}(\theta \bc) = \sum_{k=0}^\infty e^{k c_1\theta}  \sum_{\bl_{[2]}\in\scrI_n} N_{\bx,(k-\langle \bc_{[2]},\bl_{[2]} \rangle/c_1,\bl_{[2]})},  
\end{equation}
where $\bc_{[2]}=(c_2,c_3,...,c_d)$. By the Cauchy-Hadamard theorem, we obtain the following result.
\begin{theorem} \label{th:asymptotic_marginal}
For any vector of mutually prime positive integers, $\bc=(c_1,c_2,...,c_d)$, such that $c_1=\min_{1\le i\le d} c_i$ and for every $(\bx,j)\in\mathbb{S}_+$ and $j'\in S_0$, 
\begin{align}
& \limsup_{k\to\infty} \frac{1}{k} \log  \sum_{\bl_{[2]}\in\scrI_k} \tilde{q}_{(\bx,j),(k-\langle \bc_{[2]},\bl_{[2]} \rangle/c_1,\bl_{[2]},j')} 
= - \sup_{\theta \bc\in\Gamma} c_1 \theta.
\end{align}
In other cases, e.g. $c_2=\min_{1\le i\le d} c_i$, an analogous result holds. 
\end{theorem}

\begin{figure}[htbp]
\begin{center}
\includegraphics[width=11cm,trim=200 420 200 20]{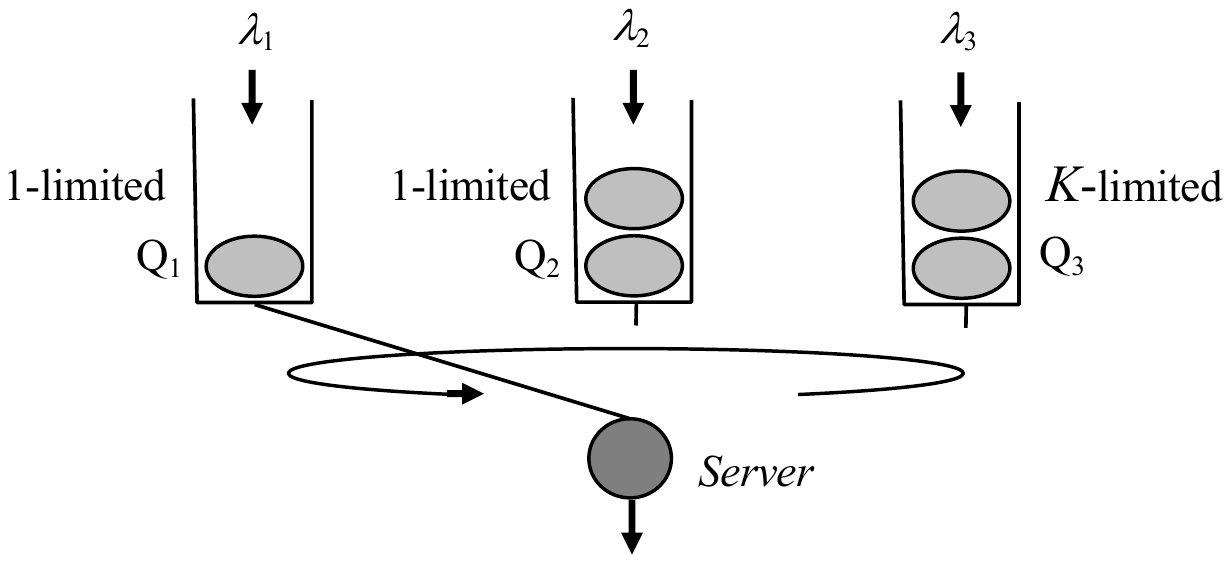} 
\caption{Polling model with three queues}
\label{fig:polling_model}
\end{center}
\end{figure}

%
%
\subsection{Single-server polling model with limited services: An example}

As a simple example, we consider a single-server polling model with three queues, in which first two queues (Q$_1$ and Q$_2$) are served according to a $1$-limited service and the other queue (Q$_3$) according to a $K$-limited service (see Fig.\ \ref{fig:polling_model}). We say that a queue is served according to a $K$-limited service if the server serves at most $K$ customers on a visit to that queue. The single server goes around the queues in order Q$_1$, Q$_2$, Q$_3$, without switchover times. 
For $i\in\{1,2,3\}$, customers arrive at Q$_i$ according to a Poisson process with intensity $\lambda_i$ and they receive exponential services with mean $1/\mu_i$. We denote by $\lambda$ the sum of the arrival rates, i.e., $\lambda=\lambda_1+\lambda_2+\lambda_3$. 
For $i\in\{1,2,3\}$, let $\tilde{X}_i(t)$ be the number of customers in Q$_i$ at time $t$ and denote by $\tilde{\bX}(t)=(\tilde{X}_1(t),\tilde{X}_2(t),\tilde{X}_3(t))$ the vector of them. Let $\tilde{J}(t)$ be the server state indicating which customer is served at time $t$. Then, $\{\tilde{\bY}(t)\}=\{(\tilde{\bX}(t),\tilde{J}(t))\}$ becomes a continuous-time three-dimensional QBD process. 
Let $S_0$ be the set of server states, which is given as $S_0=\{1,2,...,K,K+1,K+2\}$. When $\tilde{\bX}(t)>\bzero$, $\tilde{J}(t)=1$ means that the server is serving a customer in Q$_1$ and $\tilde{J}(t)=2$ that it is serving a customer in Q$_2$; for $j\ge 3$, $\tilde{J}(t)=j$ means that it is serving the $(j-2)$-th customer in Q$_3$ on a visit to that queue.  
The nonzero transition rate blocks of $\{\tilde{\bY}(t)\}$ when $\tilde{\bX}(t)>\bzero$ are given as follows:
\begin{align*}
&\tilde{A}_{1,0,0}=\lambda_1 I,\quad \tilde{A}_{0,1,0}=\lambda_2 I,\quad \tilde{A}_{0,0,1}=\lambda_3 I, \cr
&\tilde{A}_{-1,0,0} = \begin{pmatrix}
0 & \mu_1 & 0 & & 0\cr
0 & 0 & 0 &  \cdots & 0\cr
0 & 0 & 0 &  & 0\cr
& \vdots & & \ddots  & \cr
0 & 0 & 0 & \cdots & 0
\end{pmatrix},\quad 
\tilde{A}_{0,-1,0} = \begin{pmatrix}
0 & 0 & 0 & 0 & & 0\cr
0 & 0 & \mu_2 & 0 & \cdots & 0\cr
0 & 0 & 0 & 0 & & 0\cr
0 & 0 & 0 & 0 & & 0\cr
& & \vdots & & \ddots  & \cr
0 & 0 & 0 & 0 & \cdots & 0
\end{pmatrix},\cr
&\tilde{A}_{0,0,-1} = \begin{pmatrix}
0 & 0 & 0 & 0 & & 0\cr
0 & 0 & 0 & 0 & \cdots & 0\cr
0 & 0 & 0 & \mu_3 & & 0\cr
& & \vdots & & \ddots  & \cr
0 & 0 & 0 & 0 & \cdots & \mu_3 \cr
\mu_3 & 0 & 0 & 0 & \cdots & 0
\end{pmatrix},\quad
\tilde{A}_{0,0,0} 
= -\diag\!\left(\sum_{\bi\in\{-1,0,1\}^3,\,\bi\ne\bzero} \tilde{A}_{\bi} \bone \right). 
\end{align*}

Let $\{\bY(t)\}=\{(\bX(t),J(t))\}$ be a continuous-time three-dimensional MMRW on the state space $\mathbb{Z}^3\times S_0$, having $\tilde{A}_{\bi},\,\bi\in\{-1,0,1\}^3$, as the transition rate blocks. 
Let $\{\bY_n\}=\{(\bX_n,J_n)\}$ be a discrete-time three-dimensional MMRW on the state space $\mathbb{Z}^3\times S_0$, generated from $\{\bY(t)\}$ by the uniformization technique. The transition probability blocks of $\{\bY_n\}$ are given by, for $\bi\in\{-1,0,1\}^3$, 
\[
A_{\bi} = \left\{ \begin{array}{ll}
I+\frac{1}{\nu} \tilde{A}_{\bi}, & \bi=\bzero, \cr
\frac{1}{\nu} \tilde{A}_{\bi}, & \mbox{otherwise},
\end{array} \right.
\]
where we set $\nu=\lambda+\mu_1+\mu_2+\mu_3$. 
Applying Theorem \ref{th:asymptotic_any_direction} and Corollary \ref{co:asymptotic_any_direction} to this MMRW $\{\bY_n\}$, we obtain the asymptotic decay rates of the occupation measures, as described in Tables \ref{tab:table1} and \ref{tab:table2}. In both the tables, the value of $K$ varies from 1 to 20. 
Table \ref{tab:table1} deals with a symmetric case, where all the arrival intensities are set at 0.25 and all the service rates are set at 1. Due to that Q$_3$ is served according to a $K$-limited service, the absolute  value of the asymptotic decay rate in the cases where $c_3=1$ monotonically increases as the value of $K$ increases. On the other hand, that in the cases where $c_3=0$ does not always vary monotonically, for example, in the case where $\bc=(1,1,0)$, the absolute  value of the asymptotic decay rate decreases at first and then it increases. 
Table \ref{tab:table2} deals with an asymmetric case, where the arrival intensity of Q$_3$ is five times as large as those in Q$_1$ and Q$_2$, i.e., $\mu_1=\mu_2=0.1$ and $\mu_3=0.5$; all the service rates are set at 1. 
It can be seen from the table that the absolute  values of the asymptotic decay rates for all the direction vectors are nearly balanced when  $K$ is greater than 5, which means that the absolute value of the asymptotic decay rate in the case where $\bc=(1,1,0)$ is close to that in the case where $\bc=(1,0,1)$ when $K$ is set at 5; the absolute value of the asymptotic decay rate in the case where $\bc=(1,0,0)$ is close to that in the case where $\bc=(0,0,1)$ when $K$ is set at 10.

%
%
%
%

\begin{table}[htp]
\caption{Values of $\sup_{\btheta\in\Gamma} \langle \bc,\btheta \rangle$ ($\lambda_1=\lambda_2=\lambda_3=0.25$, $\mu_1=\mu_2=\mu_3=1$).}
\begin{center}
\begin{tabular}{|c|cccccc|} \hline
& \multicolumn{6}{c|}{$K$} \cr \cline{2-7}
$\bc$ & 1 & 2 & 3 & 5 & 10 & 20 \cr \hline
$(1,1,1)$ & 0.86 & 1.10 & 1.26 & 1.41 & 1.54 & 1.61 \cr
$(1,1,0)$ & 0.69 & 0.59 & 0.63 & 0.69 & 0.74 & 0.78 \cr
$(1,0,1)$ & 0.69 & 1.11 & 1.37 & 1.62 & 1.84 & 1.97 \cr
$(1,0,0)$ & 0.45 & 0.51 & 0.62 & 0.75 & 0.88 & 0.97 \cr
$(0,0,1)$ & 0.45 & 0.99 & 1.25 & 1.49 & 1.68 & 1.77 \cr \hline
\end{tabular}
\end{center}
\label{tab:table1}
\end{table}%

\begin{table}[htp]
\caption{Values of $\sup_{\btheta\in\Gamma} \langle \bc,\btheta \rangle$ ($\lambda_1=\lambda_2=0.1$, $\lambda_3=0.5$, $\mu_1=\mu_2=\mu_3=1$).}
\begin{center}
\begin{tabular}{|c|cccccc|} \hline
& \multicolumn{6}{c|}{$K$} \cr \cline{2-7}
$\bc$ & 1 & 2 & 3 & 5 & 10 & 20 \cr \hline
$(1,1,1)$ & 2.81 & 2.34 & 1.90 & 1.33 & 1.08 & 1.07 \cr
$(1,1,0)$ & 3.33 & 2.57 & 1.94 & 1.18 & 0.80 & 0.74 \cr
$(1,0,1)$ & 1.72 & 1.44 & 1.21 & 0.95 & 1.01 & 1.18 \cr
$(1,0,0)$ & 2.01 & 1.54 & 1.17 & 0.76 & 0.68 & 0.79 \cr
$(0,0,1)$ & 0.41 & 0.37 & 0.36 & 0.41 & 0.62 & 0.78  \cr \hline
\end{tabular}
\end{center}
\label{tab:table2}
\end{table}%

%
%

%
\section{Concluding remark} \label{sec:concluding}

Using the results in the paper, we can obtain lower bounds for the asymptotic decay rates of the stationary distribution in a multi-dimensional QBD process. 
Let $\{\tilde{\bY}_n\}=\{(\tilde{\bX}_n,\tilde{J}_n)\}$ be a $d$-dimensional QBD process on the state space $\mathbb{S}_+=\mathbb{Z}_+^d\times S_0$, and assume that the blocks of transition probabilities when $\tilde{\bX}_n>\bzero$ are given by  $A_{\bi},\bi\in\{-1,0,1\}^d$. 
Assume that $\{\tilde{\bY}_n\}$ is irreducible and positive recurrent and denote by $\bnu=(\nu_{\by},\by\in\mathbb{S}_+)$ the stationary distribution of the QBD process. Further assume that the blocks $A_{\bi},\bi\in\{-1,0,1\}^d,$ satisfy the property corresponding to Assumption \ref{as:P_irreducible}. 
%
%
%
%
%
%
%
Then, by Theorem \ref{th:asymptotic_any_direction} and Corollary \ref{co:asymptotic_any_direction}, for any vector $\bc$ of nonnegative integers such that $\bc\ne\bzero$ and for every $j\in S_0$, a lower bound for the asymptotic decay rate of the stationary distribution in the QBD process in the direction specified by $\bc$ is given as follows:
\begin{equation}
\liminf_{k\to\infty} \frac{1}{k} \log \nu_{(k \bc,j)} 
\ge - \sup\{ \langle \bc, \btheta \rangle; \cp(A_*(\btheta)) > 1,\,\btheta\in\mathbb{R}^d \},
\label{eq:liminf_nu}
\end{equation}
where $A_*(\btheta)=\sum_{\bi\in\{-1,0,1\}^d} e^{\langle \bi,\btheta \rangle} A_{\bi}$. 
Since the QBD process is a reflected Markov additive process, this inequality is an answer to Conjecture 5.1 of \cite{Miyazawa12} in a case with background states.

%
%

%
%
\appendix
%
\section{Convexity of the reciprocal of a convergence parameter} \label{sec:cp_convex}

Let $n$ be a positive integer and $\bx=(x_1,x_2,...,x_n)\in\mathbb{R}^n$. 
We say that a positive function $f(\bx)$ is log-convex in $\bx$ if $\log f(\bx)$ is convex in $\bx$, and denote by $\mathfrak{S}_n$ the class of all log-convex functions of $n$ variables, together with the function identically zero. Note that, $\mathfrak{S}_n$ is closed under addition, multiplication, raising to any positive power, and ``$\limsup$" operation. Furthermore, a log-convex function is a convex function. 

Let $F(\bx)=(f_{ij}(\bx),i,j\in\mathbb{Z}_+)$ be a matrix function each of whose elements belongs to the class $\mathfrak{S}_n$, i.e., for every $i,j\in\mathbb{Z}_+$, $f_{i,j}\in\mathfrak{S}_n$. In \cite{Kingman61}, it has been proved that when $n=1$ and $F(x)$ is a square matrix of a finite dimension, the maximum eigenvalue of $F(x)$ is a log-convex function in $x$. Analogously, we obtain the following lemma. 
\begin{lemma} \label{le:cp_convex}
For every $\bx\in\mathbb{R}^n$, assume all iterates of $F(\bx)$ is finite and $F(\bx)$ is irreducible. Then, the reciprocal of the convergence parameter of $F(\bx)$, $\cp(F(\bx))^{-1}$, is log-convex in $\bx$ or identically zero.
\end{lemma}
\begin{proof}
For $k\ge 0$, we denote by $f^{(k)}_{i,j}(\bx)$ the $(i,j)$-element of $F(\bx)^k$. 
First, we show that, for every $k\ge 1$ and for every $i,j\in\mathbb{Z}_+$, $f^{(k)}_{i,j}(\bx)\in\mathfrak{S}_n$. It is obvious when $k=1$. Suppose that it holds for $k$. Then, we have, for every $i,j\in\mathbb{Z}_+$,  
\begin{equation}
f^{(k+1)}_{i,j}(\bx) = \lim_{m\to\infty} \sum_{l=0}^m f^{(k)}_{i,l}(\bx)\, f_{l,j}(\bx), 
\end{equation}
and this leads us to $f^{(k+1)}_{i,j}(\bx)\in\mathfrak{S}_n$ since $\mathfrak{S}_n$ is closed under addition, multiplication and ``$\limsup$" (``$\lim$") operation. Therefore, for every $k\ge 1$, every element of $F(\bx)^n$ belongs to $\mathfrak{S}_n$. 

Next, we note that, by Theorem 6.1 of \cite{Seneta06}, since $F(\bx)$ is irreducible, all elements of the power series $\sum_{k=0}^\infty z^k F(\bx)^k$ have the common convergence radius (convergence parameter), which is denoted by $\cp(F(\bx))$. 
By the Cauchy-Hadamard theorem, we have, for any $i,j\in\mathbb{Z}_+$, 
\begin{equation}
\cp(F(\bx))^{-1} = \limsup_{k\to\infty} \bigl(f^{(k)}_{i,j}(\bx)\bigr)^{1/k}, 
\end{equation}
and this implies $\cp(F(\bx))^{-1}\in\mathfrak{S}_n$ since $\bigl(f^{(k)}_{i,j}(\bx)\bigr)^{1/k}\in\mathfrak{S}_n$ for any $k\ge 1$. 
\end{proof}

%
\section{Proof of Lemma \ref{le:RandGmatrix_equations}} \label{sec:proof_RandGmatrix_equations}

\begin{proof}
{\it (i)}\quad 
For $n\ge 1$, $\scrI_{D,1,n}$ and $\scrI_{U,1,n}$ satisfy 
\begin{align*}
\scrI_{D,1,n} &= \biggl\{\bi_{(n)}\in\{-1,0,1\}^n; \sum_{l=1}^k i_l\ge 0\ \mbox{for $k\in\{1,2,...,n-2\}$},\ \sum_{l=1}^{n-1} i_l=0\ \mbox{and}\ i_n=-1 \biggr\} \cr
&= \{(\bi_{(n-1)},-1); \bi_{(n-1)}\in\scrI_{n-1} \}, \\
\scrI_{U,1,n} &= \biggl\{\bi_{(n)}\in\{-1,0,1\}^n; i_1=1,\,\sum_{l=2}^k i_l\ge 0\ \mbox{for $k\in\{2,...,n-1\}$}\ \mbox{and} \sum_{l=2}^n i_l=0 \biggr\} \cr 
&= \{(1,\bi_{(n-1)}); \bi_{(n-1)}\in\scrI_{n-1} \},
\end{align*}
where $\bi_{(n)}=(i_1,i_2,...,i_n)$. Hence, by the Fubini's theorem, we have, for $i,j\in\mathbb{Z}_+$, 
\begin{align*}
&[G]_{i,j} = \sum_{n=1}^\infty \sum_{k=0}^\infty [Q_{0,0}^{(n-1)}]_{i,k}\, [ A_{-1}]_{k,j}  = [N  A_{-1}]_{i,j}, \\
&[R]_{i,j} = \sum_{n=1}^\infty \sum_{k=0}^\infty [ A_1]_{i,k} [Q_{0,0}^{(n-1)}]_{k,j} = [ A_1 N]_{i,j}. 
\end{align*}. 

{\it (ii)}\quad 
We prove equation (\ref{eq:Gmatrix_equation0}). 
In a manner similar to that used in (i), we have, for $n\ge 3$, 
\begin{align*}
\scrI_{D,1,n} 
&= \{(0,\bi_{(n-1)}); \bi_{(n-1)}\in\scrI_{D,1,n-1} \} \cup \{(1,\bi_{(n-1)}):\ \bi_{(n-1)}\in\scrI_{D,2,n-1} \}, \\
%
%
\scrI_{D,2,n} 
&= \bigcup_{m=1}^{n-1} \{(\bi_{(m)},\bi_{(n-m)}); \bi_{(m)}\in\scrI_{D,1,m}\ \mbox{and}\ \bi_{(n-m)}\in\scrI_{D,1,n-m} \}. 
\end{align*}
Hence, we have, for $n\ge 3$, 
\begin{align*}
D^{(n)} 
&= A_0 D^{(n-1)} + A_1 \sum_{\bi_{(n-1)}\in\scrI_{D,2,n-1}} A_{i_1} A_{i_2} \cdots A_{i_{n-1}} \cr
&= A_0 D^{(n-1)} + A_1 \sum_{m=1}^{n-1} D^{(m)} D^{(n-m-1)}, 
\end{align*} 
and by the Fubini's theorem, we obtain, for $i,j\in\mathbb{Z}_+$, 
\begin{align*}
[G]_{i,j} 
&= [ D^{(1)}]_{i,j} + \sum_{n=2}^\infty \sum_{k=0}^\infty [ A_0]_{i,k} [D^{(n-1)}]_{k,j} 
+ \sum_{n=3}^\infty \sum_{m=1}^{n-2} \sum_{k=0}^\infty \sum_{l=0}^\infty [ A_1]_{i,k} [D^{(m)}]_{k,l} [D^{(n-m-1)}]_{l,j} \cr
&= [ A_{-1}]_{i,j} + [ A_0 G]_{i,j} + [ A_1 G^2]_{i,j}, 
\end{align*}
where we use the fact that $D^{(1)}=A_{1}$ and $D^{(2)}=A_0 A_{-1}=A_0 D^{(1)}$. 
Equation (\ref{eq:Rmatrix_equation0}) can analogously be proved. 

{\it (iii)}\quad 
We prove equation (\ref{eq:NandH_relation}). 
In a manner similar to that used in (i), we have, for $n\ge 1$, 
\begin{align*}
\scrI_{n} 
%
%
&= \{(0,\bi_{(n-1)}):\ \bi_{(n-1)}\in\scrI_{n-1} \} \cr
&\qquad \cup \left(\cup_{m=2}^n \{(1,\bi_{(m-1)},\bi_{(n-m)}):\ \bi_{(m-1)}\in\scrI_{D,1,m-1}, \bi_{(n-m)}\in\scrI_{n-m} \} \right), \\
\scrI_{n} &= \{(\bi_{(n-1)},0):\ \bi_{(n-1)}\in\scrI_{n-1} \} \cr
&\qquad \cup \left(\cup_{m=0}^{n-2} \{(\bi_{(m)},1,\bi_{(n-m-1)}):\ \bi_{(m)}\in\scrI_{m}, \bi_{(n-m-1)}\in\scrI_{D,1,n-m-1} \} \right). 
\end{align*}
Hence, we have, for $n\ge 1$, 
\begin{align*}
&Q_{0,0}^{(n)} = A_0 Q_{0,0}^{(n-1)} + \sum_{m=2}^n A_1 D^{(m-1)} Q_{0,0}^{(n-m)}, \\
&Q_{0,0}^{(n)} = Q_{0,0}^{(n-1)} A_0 + \sum_{m=0}^{n-2} Q_{0,0}^{(m)} A_1 D^{(m-n-1)} , 
\end{align*}
and by the Fubini's theorem, we obtain, for $i,j\in\mathbb{Z}_+$, 
\begin{align*}
[N]_{i,j} 
&= \delta_{ij} 
+ \sum_{n=1}^\infty \sum_{k=0}^\infty [A_0]_{i,k} [Q_{0,0}^{(n-1)}]_{k,j} 
+ \sum_{n=2}^\infty \sum_{m=2}^n \sum_{k=0}^\infty \sum_{l=0}^\infty [A_1]_{i,k} [D^{(m-1)}]_{k,l} [Q_{0,0}^{(n-m)}]_{l,j} \cr
&= \delta_{i,j} + [A_0 N]_{i,j} + [A_1 G N]_{i,j}, \\
[N]_{i,j} 
&= \delta_{ij} 
+ \sum_{n=1}^\infty \sum_{k=0}^\infty [Q_{0,0}^{(n-1)}]_{i,k} [A_0]_{k,j} 
+ \sum_{n=2}^\infty \sum_{m=0}^{n-2} \sum_{k=0}^\infty \sum_{l=0}^\infty [Q_{0,0}^{(m)}]_{i,k} [A_1]_{k,l} [D^{(n-m-1)}]_{l,j}  \cr
&= \delta_{i,j} + [N A_0]_{i,j} + [N A_1 G]_{i,j},
\end{align*}
where $\delta_{i,j}$ is the Kronecker delta. This leads us to equation (\ref{eq:NandH_relation}). 
\end{proof}

%
\section{A sufficient condition ensuring $\chi(\theta)$ is unbounded} \label{sec:chi_unbounded}

\begin{proposition} \label{pr:chi_unbounded}
Assume $P$ is irreducible, then $\chi(\theta)$ is unbounded in both the directions, i.e., $\lim_{\theta\to -\infty} \chi(\theta) = \lim_{\theta\to\infty} \chi(\theta)=\infty$. 
\end{proposition}

\begin{proof}
Note that, since $P$ is irreducible, $A_*$ is also irreducible. 
For $n\ge 1$, $j\in\mathbb{Z}_+$ and $\theta\in\mathbb{R}$, $A_*(\theta)^n$ satisfies 
\begin{align}
&[ A_*(\theta)^n ]_{jj}
=\ \sum_{\bi_{(n)}\in\{-1,0,1\}^n} [ A_{i_1} A_{i_2} \times \cdots \times A_{i_n} ]_{jj}\,e^{\theta \sum_{k=1}^n i_k},
\label{eq:Asesn}
\end{align}
where $\bi_{(n)}=(i_1,i_2,...,i_n)$.
Since $P$ is irreducible, there exist $n_0>1$ and $\bi_{(n_0)}\in\{-1,0,1\}^{n_0}$ such that $[ A_{i_1} A_{i_2} \times \cdots \times A_{i_{n_0}} ]_{jj}>0$ and $\sum_{k=1}^{n_0} i_k =1$. For such a $n_0$, we have $[A_*(\theta)^{n_0}]_{jj}\ge c e^\theta$ for some $c>0$. 
This implies that, for any $m\ge 1$, $[A_*(\theta)^{n_0 m}]_{jj}\ge c^m e^{m \theta}$ and we have 
\[
\chi(\theta)
= \limsup_{m\to\infty} ([A_*(\theta)^m]_{jj})^{\frac{1}{m}} 
\ge \limsup_{m\to\infty} ([A_*(\theta)^{n_0 m}]_{jj})^{\frac{1}{n_0 m}}
\ge c^{\frac{1}{n_0}} e^{\frac{\theta}{n_0}}. 
\]
Therefore, $\lim_{\theta\to\infty} \chi(\theta)=\infty$. 
Analogously, we can obtain $\chi(\theta)\ge c^{\frac{1}{n_0}} e^{-\frac{\theta}{n_0}}$ for some $n_0\ge 1$ and $c>0$, and this implies that $\lim_{\theta\to -\infty} \chi(\theta)=\infty$. 
\end{proof}

%
\section{Proof of Proposition \ref{pr:R_u} and Corollary \ref{co:cpRlimit}} \label{sec:proof_R_u}

%
%
\begin{proof}[Proof of Proposition \ref{pr:R_u}]
Let $S_1$ be the set of indexes of nonzero rows of $A_1$, i.e., $S_1=\{k\in\mathbb{Z}_+; \mbox{the $k$-th row of $A_1$ is nonzero} \}$, and $S_2=\mathbb{Z}_+\setminus S_1$.  For $i\in\{-1,0,1\}$, reorder the rows and columns of $A_i$ so that it is represented as
\[
A_i =\begin{pmatrix}
A_{i,11} & A_{i,12} \cr
A_{i,21} & A_{i,22}
\end{pmatrix},
\]
where $A_{i,11}=( [A_i]_{k,l}; k,l\in S_1 )$, $A_{i,12}=( [A_i]_{k,l}; k\in S_1, l\in S_2 )$, $A_{i,21}=( [A_i]_{k,l}; k\in S_2, l\in S_1 )$ and $A_{i,22}=( [A_i]_{k,l}; k,l\in S_2 )$. By the definition of $S_1$, every row of $\begin{pmatrix} A_{1,11} & A_{1,12} \end{pmatrix}$ is nonzero and we have $A_{1,21}=O$ and $A_{1,22}=O$. 
Since $Q$ is irreducible and $R$ is finite, $N$ is also finite and positive. Hence, $R$ is given as
\begin{equation}
R = A_1 N 
= \begin{pmatrix}
R_{11} & R_{12} \cr
O & O
\end{pmatrix},
\label{eq:R11R12}
\end{equation}
where $R_{11}=( [R]_{k,l}; k,l\in S_1 )$ is positive and hence irreducible; $R_{12}=( [R]_{k,l}; k\in S_1, l\in S_2 )$ is also positive. 
Since $R_{11}$ is a submatrix of $R$, we have $\cp(R_{11})\ge \cp(R)$. 

We derive an inequality with respect to $R_{11}$ and $R_{12}$. From (\ref{eq:Rmatrix_equation0}), we obtain $R\ge R^2 A_{-1} + R A_0$ and, from this inequality, 
\begin{align}
&R_{11} \ge R_{11} R_{12} A_{-1,21} + R_{12} A_{0,21}, \label{eq:R11_ineq} \\
&R_{12} \ge R_{11} R_{12} A_{-1,22} + R_{12} A_{0,22}. \label{eq:R12_ineq}
\end{align}
For $n\ge 1$ and $\bi_{(n)}=(i_1,i_2,...,i_n)\in\{-1,0\}^n$, define $A_{\bi_{(n)},22}$ and $\Vert \bi_{(n)} \Vert$ as
\[
A_{\bi_{(n)},22} = A_{i_n,22}\times A_{i_{n-1},22}\times \cdots \times A_{i_1,22},\quad 
\Vert \bi_{(n)} \Vert = \sum_{k=1}^n |i_k|.
\] 
Then, by induction using (\ref{eq:R12_ineq}), we obtain, for $n\ge 1$, 
\begin{equation}
R_{12} \ge \sum_{\bi_{(n)}\in\{-1,0\}^n} R_{11}^{\Vert \bi_{(n)} \Vert} R_{12} A_{\bi_{(n)},22}, 
\label{eq:R12_ineq2}
\end{equation}
and this and (\ref{eq:R11_ineq}) lead us to, for $n\ge 1$,  
\begin{equation}
R_{11} \ge \sum_{\bi_{(n)}\in\{-1,0\}^n} R_{11}^{\Vert \bi_{(n)} \Vert} R_{12} A_{\bi_{(n-1)},22} A_{i_n,21}, 
\label{eq:R11_ineq2}
\end{equation}
where $A_{\bi_{(0)},22} = I$. 
We note that since $A_*=A_{-1}+A_0+A_1$ is irreducible, $A_{1,21}=O$ and $A_{1,22}=O$, for every $k\in S_2$ and $l\in S_1$, there exist $n_0\ge 1$ and $\bi_{(n_0)}\in\{-1,0\}^{n_0}$ such that $[A_{\bi_{(n_0-1)},22} A_{i_{n_0},21}]_{k,l}>0$. 

Let $\alpha$ be the convergence parameter of $R_{11}$. Since $R_{11}$ is irreducible, $R_{11}$ is either $\alpha$-recurrent or $\alpha$-transient. 
First, we assume $R_{11}$ is $\alpha$-recurrent. Then, there exists a positive vector $\bu_1$ such that $\alpha \bu_1^\top R_{11}=\bu_1^\top$. If $\bu_1^\top R_{12}<\infty$, then $\bu^\top=(\bu_1^\top, \alpha \bu_1^\top R_{12})$ satisfies $\alpha \bu^\top R = \bu^\top$ and we obtain $\cp(R)\ge\alpha=\cp(R_{11})$. Since $\cp(R)\le\cp(R_{11})$, this implies $\alpha=\cp(R)=e^{\bar{\theta}}$ and we obtain statement (i) of the proposition. We, therefore, prove $\bu_1^\top R_{12}<\infty$. 
Suppose, for some $k\in S_2$, the $k$-th element of $\bu_1^\top R_{12}$ diverges. For this $k$ and any $l\in S_1$, there exist $n_0\ge 1$ and $\bi_{(n_0)}\in\{-1,0\}^{n_0}$ such that $[A_{\bi_{(n_0-1)},22} A_{i_{n_0},21}]_{k,l}>0$. Hence, from (\ref{eq:R11_ineq2}), we obtain 
\begin{align*}
[\alpha^{-1} \bu_1^\top]_l 
= [\bu_1^\top R_{11}]_l 
&\ge [\bu_1^\top R_{11}^{\Vert \bi_{(n_0)} \Vert} R_{12}]_k [A_{\bi_{(n_0-1)},22} A_{i_{n_0},21}]_{k,l}  \cr
&= \alpha^{-\Vert \bi_{(n_0)} \Vert} [\bu_1^\top R_{12}]_k [A_{\bi_{(n_0-1)},22} A_{i_{(n_0)},21}]_{k,l}.
\end{align*}
This contradicts $\bu_1$ is finite and we see $\bu_1^\top R_{12}$ is finite. 

Next, we assume $R_{11}$ is $\alpha$-transient, i.e., $\sum_{n=0}^\infty \alpha^n R_{11}^n<\infty$. We have 
\[
\sum_{n=0}^\infty \alpha^n R^n 
=\begin{pmatrix}
\sum_{n=0}^\infty \alpha^n R_{11}^n & \sum_{n=1}^\infty \alpha^n R_{11}^{n-1} R_{12} \cr
O & O
\end{pmatrix}. 
\]
Hence, in order to prove $\sum_{n=0}^\infty \alpha^n R^n<\infty$, it suffices to demonstrate $\sum_{n=1}^\infty \alpha^n R_{11}^{n-1} R_{12}<\infty$. 
Suppose, for some $k\in S_1$ and some $l\in S_2$, the $(k,l)$-element of $\sum_{n=1}^\infty \alpha^n R_{11}^{n-1} R_{12}$ diverges. For this $l$ and any $m\in S_1$, there exist $n_0\ge 1$ and $\bi_{(n_0)}\in\{-1,0\}^{n_0}$ such that $[A_{\bi_{(n_0-1)},22} A_{i_{n_0},21}]_{l,m}>0$. 
For such an $n_0$ and $\bi_{(n_0)}$, we obtain from (\ref{eq:R11_ineq2}) that, for $n\ge 1$,  
\[
R_{11}^n = R_{11}^{n-1} R_{11} \ge R_{11}^{\Vert \bi_{(n_0)} \Vert} R_{11}^{n-1} R_{12} A_{\bi_{(n_0-1)},22} A_{i_{n_0},21}, 
\]
where every diagonal element of $R_{11}^{\Vert \bi_{(n_0)} \Vert}$ is positive. From this inequality, we obtain 
\[
\left[\sum_{n=1}^\infty \alpha^n R_{11}^n\right]_{k,m} 
\ge \left[ R_{11}^{\Vert \bi_{(n_0)} \Vert} \right]_{k,k} \left[ \sum_{n=1}^\infty \alpha^n R_{11}^{n-1} R_{12} \right]_{k,l} \left[ A_{\bi_{(n_0-1)},22} A_{i_{n_0},21} \right]_{l,m}.
\]
This contradicts $R_{11}$ is $\alpha$-transient and we obtain $\sum_{n=0}^\infty \alpha^n R^n<\infty$. Furthermore, this leads us to $\cp(R)\ge \alpha=\cp(R_{11})$ and, from this and $\cp(R)\le\cp(R_{11})$, we have $\alpha=\cp(R)=e^{\bar{\theta}}$. As a result, we obtain statement (ii) of the proposition and this completes the proof. 
\end{proof}

%
%
\begin{proof}[Proof of Corollary \ref{co:cpRlimit}]
In a manner similar to that used in the proof of Proposition \ref{pr:R_u}, let $S_1$ be the set of indexes of nonzero rows of $A_1$ and $S_2=\mathbb{Z}_+\setminus S_1$. Then, reordering the rows and columns of $R$  according to $S_1$ and $S_2$, we obtain $R$ given by expression (\ref{eq:R11R12}), where $R_{11}=( [R]_{k,l}; k,l\in S_1 )$ is positive and hence irreducible and $R_{12}=( [R]_{k,l}; k\in S_1, l\in S_2 )$ is also positive. From the proof of Proposition \ref{pr:R_u},  we know that $ \cp(R)=\cp(R_{11})=e^{\bar{\theta}}$. 
By these facts and the Cauchy-Hadamard theorem, we obtain, for $i,j\in S_1$, $k\in S_2$ and $n\ge 1$, 
\begin{align}
&\limsup_{n\to\infty} \left( [R^n]_{i,j} \right)^{\frac{1}{n}}
= \limsup_{n\to\infty} \left( [R_{11}^n]_{i,j} \right)^{\frac{1}{n}}
= e^{-\bar{\theta}}, \label{eq:limsupR11} \\
&\limsup_{n\to\infty} \left( [R^n]_{i,k} \right)^{\frac{1}{n}} \le e^{-\bar{\theta}}.
\end{align}
Since $[R_{11}^n]_{i,i}$ is subadditive with respect to $n$, i.e., $[R_{11}^{n_1+n_2}]_{i,i}\ge [R_{11}^{n_1}]_{i,i}\, [R_{11}^{n_2}]_{i,i}$ for $n_1,n_2\in\mathbb{Z}_+$, the limit sup in equation (\ref{eq:limsupR11}) can be replaced with the limit when $i=j$ (see, e.g., Lemma A.4 of \cite{Seneta06}). 
Furthermore, we have, for $i,j\in S_1$, $k\in S_2$ and $n\ge 1$, 
\begin{align}
&\liminf_{n\to\infty} \left( [R^n]_{i,j} \right)^{\frac{1}{n}} 
= \liminf_{n\to\infty} \left( [R_{11}^{n-1} R_{11}]_{i,j} \right)^{\frac{1}{n}} 
\ge \liminf_{n\to\infty} \left( [R_{11}^{n-1}]_{i,i} [R_{11}]_{i,j} \right)^{\frac{1}{n}} 
= e^{-\bar{\theta}}, \\
&\liminf_{n\to\infty} \left( [R^n]_{i,k} \right)^{\frac{1}{n}} 
= \liminf_{n\to\infty} \left( [R_{11}^{n-1} R_{12}]_{i,k} \right)^{\frac{1}{n}} 
\ge \liminf_{n\to\infty} \left( [R_{11}^{n-1}]_{i,i} [R_{12}]_{i,k} \right)^{\frac{1}{n}} 
= e^{-\bar{\theta}}.
\end{align}
Hence, we obtain equation (\ref{eq:cpRlimit}). It is obvious by expression (\ref{eq:R11R12}) that, for $i\in S_2$ and $j\in\mathbb{Z}_+$, $[R^n]_{i,j}=0$. 
\end{proof}

\end{document}